\documentclass[a4paper]{article}
\usepackage{latexsym,amsfonts,amssymb,amsmath,amsthm}

\newtheorem{theorem}{Theorem}[section]
\newtheorem{lemma}{Lemma}[section]
\newtheorem{proposition}{Proposition}[section]
\newtheorem{definition}{Definition}[section]

\newtheorem{remark}{Remark}[section]

\numberwithin{equation}{section}

\numberwithin{definition}{section}

\def\be{\begin{equation}}
\def\ee{\end{equation}}

\def\p{\partial}
\def\disp{\displaystyle}

\newcommand{\JM}{Mierczy\'nski}
\newcommand{\const}{\ensuremath{\mathrm{const}}}
\newcommand{\RR}{\ensuremath{\mathbb{R}}}
\newcommand{\QQ}{\ensuremath{\mathbb{Q}}}
\newcommand{\ZZ}{\ensuremath{\mathbb{Z}}}
\newcommand{\NN}{\ensuremath{\mathbb{N}}}
\newcommand{\PP}{\ensuremath{\mathbb{P}}}
\newcommand{\lambdainf}{\ensuremath{\lambda_{\mathrm{inf}}}}
\newcommand{\lambdasup}{\ensuremath{\lambda_{\mathrm{sup}}}}
\newcommand{\OFP}{(\Omega,\mathcal{F},\PP)}
\newcommand{\Vptheta}{\ensuremath{V_{p}^{\theta}}}
\newcommand{\abs}[1]{\ensuremath{\lvert#1\rvert}}
\newcommand{\floor}[1]{\ensuremath{\lfloor#1\rfloor}}
\newcommand{\norm}[1]{\ensuremath{\lVert#1\rVert}}

\newcommand{\normVptheta}[1]{\ensuremath{\lVert#1\rVert_{\Vptheta}}}
\DeclareMathOperator{\cl}{cl}

\DeclareMathOperator{\spanned}{span}

\DeclareMathOperator{\nullspace}{\mathcal{N}}

\DeclareMathOperator{\supp}{supp}

\begin{document}

\title{Time averaging for nonautonomous/random linear parabolic equations
\footnote{This is a pre-copy-editing, author-produced PDF of an article accepted for publication in \emph{Discrete and Continuous Dynamical Systems Series B} following peer review. The definitive publisher-authenticated version \emph{Discrete and Continuous Dynamical Systems Series B} \textbf{9}(3/4) (2008), pp. 661--699, is available online at:  http://dx.doi.org/10.3934/dcdsb.2008.9.661}}

\author{Janusz Mierczy\'nski 
\thanks{Supported by the research funds for 2005--2008 (grant MENII 1 PO3A 021 29, Poland) and NSF grant INT-0341754.}
\\
Institute of Mathematics\\
Wroc{\l}aw University of Technology\\
Wybrze\.ze Wyspia\'nskiego 27\\
PL-50-370 Wroc{\l}aw\\
Poland\\
\\
and
\\
\\
Wenxian Shen 
\thanks{Partially supported by NSF grants DMS-0504166 and  INT-0341754.}
\\
Department of Mathematics\\
Auburn University\\
Auburn University, AL 36849\\
USA
}
\date{}

\maketitle

\begin{abstract}
Linear nonautonomous/random parabolic partial differential equations
are considered under the Dirichlet, Neumann or Robin boundary conditions, where both the zero order coefficients in the equation and the coefficients in the boundary conditions are allowed to depend on time.  The theory of the principal spectrum/principal Lyapunov exponents is shown to apply to those equations.  In the nonautonomous case, the main result states that the principal eigenvalue of any time-averaged equation is not larger than the supremum of the principal spectrum and that there is a time-averaged
equation whose principal eigenvalue is not larger than the infimum
of the principal spectrum.  In the random case, the main result
states that the principal eigenvalue of the time-averaged equation
is not larger than the principal Lyapunov exponent.
\end{abstract}



\section{Introduction}
\label{introduction}
It is well known that parabolic equations can be used to model many
evolution processes in science and engineering. Parabolic equations
with general time dependence  are gaining more and more attention
since they can take various time variations of the underlying
processes into account in modeling the processes. A great amount of
research work has been carried out toward the existence, uniqueness,
and regularity of solutions of general linear, semilinear,
quasilinear parabolic equations (see \cite{Ama}, \cite{Ama1},
\cite{Dan2}, \cite{DL}, \cite{Fri1}, \cite{LaSoUr}, \cite{Lie},
\cite{Yag}, etc.). As a basic tool for nonlinear problems, it is of
great significance to study the spectral theory for linear parabolic
equations.

Spectral theory, in particular, principal spectrum theory (i.e.,\
principal eigenvalues and principal eigenfunctions theory) for time
independent and time periodic parabolic equations is well understood
(see, for example, \cite{Hes}). For such an equation, its principal
eigenvalue provides the growth rate of the evolution operator and
hence a least upper bound of the growth rates of all the solutions.
Recently much effort has been devoted to the extension of principal
eigenvalue and principal eigenfunction theory of time independent and
periodic parabolic equations to general time dependent and random
parabolic equations. See, for example, \cite{Hu1}, \cite{Hu2},
\cite{HuPo}, \cite{HuPoSa1}, \cite{HuPoSa2}, \cite{HuShVi},
\cite{Mi1}, \cite{Mi2}, \cite{MiSh1}, \cite{Po}, \cite{PoTer},
\cite{ShVi}, etc.

In the current paper, we focus on time dependent parabolic equations
of the form
\begin{equation}
\label{nonauton-eq1}
\begin{cases}
\disp\frac{\p u}{\p t} = \sum_{i,j=1}^{N} a_{ij}(x)\frac{\p^{2}u}{\p
x_{i} \p x_{j}} + \sum_{i=1}^{N}
a_i(x)\frac{\p u}{\p x_i} + c(t,x)u, \quad & t > 0,\ x \in D, \\[1.5ex]
\mathcal{B}(t)u = 0, \quad & t > 0,\ x \in \p D,
\end{cases}
\end{equation}
where $D\subset \RR^N$,
\begin{equation*}
\mathcal{B}(t)u =
\begin{cases}
u & \text{(Dirichlet)}
\\[1.5ex]
\disp \sum_{i=1}^{N} b_i(x)\frac{\p u}{\p x_i} & \text{(Neumann)} \\[1.5ex]
\disp \sum_{i=1}^{N} b_i(x)\frac{\p u}{\p x_i} + d(t,x)u &
\text{(Robin)},
\end{cases}
\end{equation*}
and random parabolic equations of the form
\begin{equation}
\label{random-eq1}
\begin{cases}
\disp\frac{\p u}{\p t} = \sum_{i,j=1}^{N} a_{ij}(x)\frac{\p^{2}u}{\p
x_{i} \p x_{j}} + \sum_{i=1}^{N} a_i(x) \frac{\p u}{\p x_i} +
c(\theta_{t}\omega,x)u, \quad & t > 0,\ x \in D,
\\[1.5ex]
\mathcal{B}(\theta_t\omega)u = 0, \quad & t > 0,\ x \in \p D,
\end{cases}
\end{equation}
where
\begin{equation*}
\mathcal{B}(\theta_t\omega)u =
\begin{cases}
u & \text{(Dirichlet)}
\\[1.5ex]
\disp
\sum_{i=1}^{N} b_i(x) \frac{\p u}{\p x_i} & \text{(Neumann)} \\[1.5ex]
\disp \sum_{i=1}^{N} b_i(x) \frac{\p u}{\p x_i} +
d(\theta_{t}\omega,x)u & \text{(Robin)},
\end{cases}
\end{equation*}
and $(\OFP, \{\theta_t\}_{t\in\RR})$ is an ergodic metric dynamical
system (see Section 2 for definition).

Our objective is to study the influence of time variations of the
zeroth order terms on the so-called principal spectrum and principal
Lyapunov exponent of \eqref{nonauton-eq1} and \eqref{random-eq1}
(which are analogs of principal eigenvalue of time independent and
periodic parabolic equations), respectively. To do so,
 we  first study the existence and uniqueness of globally positive
solutions via the skew-product semiflows on (a subspace of)
$C^1(\bar{D})$ generated by \eqref{nonauton-eq1} and
\eqref{random-eq1}.  Next we define the principal spectrum and
principal Lyapunov exponent of \eqref{nonauton-eq1} and
\eqref{random-eq1} in terms of the globally positive solutions. We
then compare the principal spectrum and principal Lyapunov exponent
of \eqref{nonauton-eq1} and \eqref{random-eq1} with those of their
averaged equations.

To be  more precise, we first introduce some notations and state some
basic assumptions.

In the Dirichlet or Neumann type boundary conditions we assume
$d(\cdot,\cdot) \equiv 0$.  In the case of~\eqref{random-eq1} we
write $c^{\omega}(t,x)$ for $c(\theta_{t}\omega,x)$, $\omega \in
\Omega$, $t \in \RR$, $x \in \bar{D}$, and $d^{\omega}(t,x)$ for
$d(\theta_{t}\omega,x)$, $\omega \in \Omega$, $t \in \RR$, $x \in \p
D$.

For $m_1, m_2 \in \NN \cup \{0\}$ and $\beta \in [0,1)$ the symbol
$C^{m_1+\beta, m_2+\beta}(\RR \times \bar{D})$ denotes the Banach
space consisting of functions $h \colon \RR \times \bar{D} \to \RR$
whose mixed derivatives of order up~to $m_1$ in $t$ and up~to $m_2$
in $x$ are bounded, and whose mixed derivatives of order $m_1$ in $t$
and $m_2$ in $x$ are globally H\"older continuous with exponent
$\beta$, uniformly in $(t,x) \in \RR \times \bar{D}$ (provided that
the boundary $\p D$ of $D$ is of class $C^{m_2+\beta}$, at~least).

Similarly, for $m_1, m_2 \in \NN \cup \{0\}$ and $\beta \in [0,1)$
the symbol $C^{m_1+\beta, m_2+\beta}(\RR \times \p D)$ denotes the
Banach space consisting of functions $h \colon \RR \times \p D \to
\RR$ whose mixed derivatives of order up~to $m_1$ in $t$ and up~to
$m_2$ in $x$ are bounded, and whose mixed derivatives of order $m_1$
in $t$ and $m_2$ in $x$ are globally H\"older continuous with
exponent $\beta$, uniformly in $(t,x) \in \RR \times \p D$ (provided
that $\p D$ is of class $C^{m_2+\beta}$, at~least).

Throughout the paper, we assume the following smoothness conditions
on the domain and the coefficients in \eqref{nonauton-eq1} and
\eqref{random-eq1} (the nonsmooth case will be considered in the
monograph \cite{MiSh2}).

\begin{itemize}
\item[(A1)]
$D \subset \RR^{N}$ is a bounded domain, with boundary $\p D$ of
class $C^{3+\alpha}$, for some $\alpha > 0$.
\item[(A2)]
The functions $a_{ij}$, $a_{i}$ belong to $C^{2}(\bar{D})$ and
the functions $b_i$ belong to $C^2(\p D)$.
\item[(A3)]
\begin{itemize}
\item[(a)]
$c \in C^{2+\alpha,1+\alpha}(\RR \times \bar{D})$ (in the
case of~\eqref{nonauton-eq1}),
\item[(b)]
$c^{\omega} \in C^{2+\alpha,1+\alpha}(\RR \times \bar{D})$
for all $\omega \in \Omega$, with the
$C^{2+\alpha,1+\alpha}(\RR \times \bar{D})$-norm bounded
uniformly in $\omega \in \Omega$ (in the case
of~\eqref{random-eq1}).
\end{itemize}
\item[(A4)]
\begin{itemize}
\item[(a)]
$d \in C^{2+\alpha,3+\alpha}(\RR \times \p D)$ (in the case
of~\eqref{nonauton-eq1}),
\item[(b)]
$d^{\omega} \in C^{2+\alpha,3+\alpha}(\RR \times \p D)$ for
all $\omega \in \Omega$, with the $C^{2+\alpha,3+\alpha}(\RR
\times \p D)$-norm bounded uniformly in $\omega \in \Omega$
(in the case of~\eqref{random-eq1}).
\end{itemize}
\end{itemize}

We also assume the following uniform ellipticity condition and the
complementing boundary condition:
\begin{itemize}
\item[(A5)]
$a_{ij}(x) = a_{ji}(x)$  for $i,j = 1,2,\dots,N$ and $x \in \bar
D$, and there is $\alpha_0 > 0$ such that
\begin{equation*}
\sum_{i,j=1}^{N} a_{ij}(x)\,\xi_i\,\xi_j \ge \alpha_{0}
\sum_{i=1}^{N}\xi_i^2, \quad x \in \bar{D}, \ \xi \in \RR^{N}.
\end{equation*}
\item[(A6)]
There is $\alpha_1 > 0$ such that
\begin{equation*}
\sum_{i=1}^{N} b_{i}(x)\nu_{i}(x) \ge \alpha_{1}, \quad x \in \p D,
\end{equation*}
where $\nu(x) = (\nu_1(x),\nu_2(x),\cdots, \nu_N(x))$ is the unit
outer normal vector of $\p D$ at $x \in \p D$.
\end{itemize}

In the case of~\eqref{nonauton-eq1} let
\begin{equation}
\label{Y(a)-def}
Y(c,d) := \cl\{\,(c,d) \cdot t:t \in \RR\,\},
\end{equation}
be equipped with the open-compact topology, where $((c, d) \cdot
t)(s,x):= (c(s+t,x),d(s+t,x))$, $s \in \RR$, $x \in \p D$, and the
closure is taken in the open-compact topology of $\RR\times \bar D$.

In the case of~\eqref{random-eq1} let
\begin{equation}
\label{Y(omega)-def}
Y(\Omega) := \cl\{\,(c^{\omega},d^{\omega}):\omega \in \Omega\,\}
\end{equation}
be equipped with the open-compact topology, where  the closure is
also taken in the open-compact topology.  We will write $Y$ instead
of $Y(c,d)$ (for the case of \eqref{nonauton-eq1}) or instead of
$Y(\Omega)$ (for the case of \eqref{random-eq1}).

For given $(\tilde c,\tilde d) \in Y$ and $u_0 \in L_p(D)$, consider
\begin{equation}
\label{nonauton-eq2}
\begin{cases}
\disp\frac{\p u}{\p t} = \sum_{i,j=1}^{N} a_{ij}(x)\frac{\p^{2}u}{\p
x_{i} \p x_{j}} + \sum_{i=1}^{N}
a_i(x)\frac{\p u}{\p x_i} + \tilde c(t,x)u, \quad & t > 0,\ x \in D, \\[1.5ex]
\tilde{\mathcal{B}}(t)u = 0, \quad & t > 0,\ x \in \p D,
\end{cases}
\end{equation}
where
\begin{equation*}
\tilde{\mathcal{B}}(t)u =
\begin{cases}
u & \text{(Dirichlet)}
\\[1.5ex]
\disp \sum_{i=1}^{N} b_i(x)\frac{\p u}{\p x_i} & \text{(Neumann)} \\[1.5ex]
\disp \sum_{i=1}^{N} b_i(x)\frac{\p u}{\p x_i} + \tilde d(t,x)u &
\text{(Robin)},
\end{cases}
\end{equation*}
with the initial condition
\begin{equation}
\label{initial-condition}
u(0,x) = u_0(x), \quad x \in D.
\end{equation}
Applying the theory presented by H. Amann in~\cite{Ama}, we have that
\eqref{nonauton-eq2}+\eqref{initial-condition} has a unique {\em
$L_{p}(D)$-solution\/} $U_{(\tilde{c},\tilde{d}),p}(\cdot,0)u_0
\colon [0,\infty) \to L_{p}(D)$  ($p > 1$) (see Proposition
\ref{existence-prop}).

Note that $U_{(\tilde{c},\tilde{d}),p}(\cdot,0)u_0$ is also a
classical solution of \eqref{nonauton-eq2}+\eqref{initial-condition}
(see Section~\ref{skew-product} for more detail).  We may therefore
write $U_{(\tilde{c},\tilde{d}),p}(t,0)u_0$ as
$U_{(\tilde{c},\tilde{d})}(t,0)u_0$ for $u_0 \in L_p(D)$.  In the
present paper we further assume the following continuous dependence.

\medskip

\begin{itemize}
\item[(A7)]
For any $T > 0$ the mapping
\begin{equation*}
[\, Y \ni (\tilde{c},\tilde{d}) \mapsto [\,[0,T] \ni t \mapsto
U_{(\tilde{c},\tilde{d})}(t,0)\,] \in
B([0,T],\mathcal{L}(L_2(D),L_2(D))) \,]
\end{equation*}
is continuous, where $\mathcal{L}(L_2(D),L_2(D))$ represents the
space of all bounded linear operators from $L_2(D)$ into itself,
endowed with the norm topology, and $B(\cdot,\cdot)$ stands for
the Banach space of bounded functions, endowed with the supremum
norm.
\end{itemize}

\medskip
It should be pointed out that in~\cite{Ama2} and \cite{Pen}
conditions, for some special cases (for example, the Dirichlet
boundary condition case and the case with infinitely differentiable
coefficients), are given that guarantee the continuous dependence of
$[\,[0,T] \ni t \mapsto U_{(\tilde{c},\tilde{d})}(t,0)\,] \in
B([0,T],\mathcal{L}(L_2(D),L_2(D)))$ on the coefficients. For the
general case, the continuous dependence of $[\,[0,T] \ni t \mapsto
U_{(\tilde{c},\tilde{d})}(t,0)\,] \in
B([0,T],\mathcal{L}(L_2(D),L_2(D)))$ on the coefficients is not
covered in \cite{Ama2} and \cite{Pen}. We will not investigate the
conditions under which (A7) is satisfied in this paper.

\medskip
Then \eqref{nonauton-eq1} (\eqref{random-eq1}) generates the
following  skew-product semiflow (see Section~\ref{skew-product} for
detail)
\begin{equation*}
\Pi_t \colon X \times Y \to X \times Y,
\end{equation*}
\begin{equation*}
\Pi_t(u_0,(\tilde{c},\tilde{d})) =
(U_{(\tilde{c},\tilde{d})}(t,0)u_0,(\tilde{c},\tilde{d}) \cdot t),
\end{equation*}
where
\begin{equation}
\label{X-eq}
X :=
\begin{cases}
\overset{\circ}{C^{1}}(\bar{D}) & \quad \text{(Dirichlet)} \\[2ex]
C^{1}(\bar{D}) & \quad \text{(Neumann or Robin),}
\end{cases}
\end{equation}
\begin{equation*}
\overset{\circ}{C^{1}}(\bar{D}) := \{\,u \in C^{1}(\bar{D}): u(x) =
0 \text{ for each } x \in \p D\,\}.
\end{equation*}
\medskip

Throughout the paper, we denote $\norm{\cdot}$ as the norm in
$L_2(D)$ (see Section 2 for other notations).

\medskip
Among others, we prove
{\em
\begin{itemize}
\item[1)]
$\Pi_t$ is strongly monotone \textup{(}see Theorem
\ref{strong-positivity-class}\textup{)}.
\item[2)]
\eqref{nonauton-eq2} has a unique \textup{(}up~to multiplication
by positive scalars\textup{)} globally positive solution
$v(t,x;\tilde{c},\tilde{d})$ \textup{(}which is an analog of a
principal eigenfunction\textup{)} \textup{(}see Theorem
\ref{globally-positive-existence}\textup{)} \textup{(}we denote
$v((\tilde{c},\tilde{d}))(\cdot)$ as
$v(0,\cdot;\tilde{c},\tilde{d})/
\norm{v(0,\cdot;\tilde{c},\tilde{d})}$\textup{)}.
\item[3)]
Consider \eqref{nonauton-eq1}. Then the set $\Sigma(c,d)$
consisting of all limits
\begin{equation}
\label{nonauton-principal-spectrum-eq0}
\lim\limits_{n\to\infty} \frac{\ln\norm{U_{(c,d) \cdot
S_n}(T_n-S_n,0)v((c,d) \cdot S_n)}} {T_n-S_n} \end{equation} where
$T_n - S_n \to \infty$ as $n \to \infty$,  is a compact interval
\textup{(}see Theorem \ref{principal-spectrum-thm2}\textup{)}.
\item[4)]
Consider \eqref{random-eq1}. Then for a.e. $\omega \in \Omega$
\begin{equation}
\label{random-lyapunov-exponent-eq0}
\lim\limits_{T\to\infty}\frac{\ln\norm{U_\omega(T,0)v(\omega)}}{T} =
\const
\end{equation}
where $U_\omega (t,0) = U_{(c^\omega,d^\omega)}(t,0)$ and
$v(\omega) = v((c^\omega,d^\omega))$ \textup{(}see Theorem
\ref{principal-exponent-thm}\textup{)}.
\end{itemize}
}

\medskip
Denote the compact interval in 3) by $[\lambdainf(c,d),
\lambdasup(c,d)]$ and  the constant in 4) by $\lambda(c,d)$. We call
$[\lambdainf(c,d),\lambdasup(c,d)]$ the {\em principal spectrum\/} of
\eqref{nonauton-eq1} (see Definition \ref{principal-spectrum-def})
and call $\lambda(c,d)$ the {\em principal Lyapunov exponent\/} of
\eqref{random-eq1} (see Definition \ref{principal-exponent-def}).

Observe that if $c(t,x)$ and $d(t,x)$ in \eqref{nonauton-eq1} are
independent of $t$ or are periodic in $t$, then $\lambdainf(c,d)( =
\lambdasup(c,d))$ is the principal eigenvalue of \eqref{nonauton-eq1}
and $v(t,\cdot;c,d)$ is an eigenfunction associated with
$\lambdainf(c,d)$ (called a principal eigenfunction).  As in the time
independent and periodic cases, the principal spectrum of
\eqref{nonauton-eq1} and principal Lyapunov exponent of
\eqref{random-eq1} provide upper bounds of growth rates of the
solutions of \eqref{nonauton-eq1} and \eqref{random-eq1},
respectively.  This can indeed  be easily seen from the fact that
\begin{align*}
\lim_{n\to\infty} \frac{\ln\norm{U_{(c,d) \cdot
S_n}(T_n-S_n,0)v((c,d) \cdot S_n)}} {T_n-S_n} & =
\lim_{n\to\infty}\frac{\ln\norm{U_{(c,d)\cdot
S_n}(T_n-S_n,0)}}{T_n-S_n} \\
& = \lim_{n\to \infty}\frac{\ln\norm{U_{(c,d)\cdot
S_n}(T_n-S_n,0)u_0}}{T_n-S_n}
\end{align*}
for any nontrivial $u_0\in X$ with $u_0(x) \ge 0$ for $x \in D$ as
long as the limits exist (the existence of one of the limits implies
the existence of the others), and
\begin{align*}
\lim_{T\to\infty} \frac{\ln\norm{U_\omega(T,0)v(\omega)}}{T} & =
\lim_{T\to\infty} \frac{\ln\norm{U_{\omega}(T,0)}}{T} \\
& = \lim_{T\to\infty} \frac{\ln\norm{U_{\omega}(T,0)u_0}}{T}
\end{align*}
for any nontrivial $u_0 \in X$ with $u_0(x) \ge 0$ for $x \in D$ as
long as the limits exist (again the existence of one of the limits
implies the existence of the others) (this fact follows from Theorem
\ref{globally-positive-existence}).

We remark that  the existence and uniqueness of globally positive
solutions to nonautonomous parabolic equations with time independent
boundary conditions were studied in \cite{Mi1}, \cite{Mi2},
\cite{Po}.  In \cite{Hu1} the author studied the uniqueness of
globally positive solutions to nonautonomous parabolic equations with
time dependent boundary conditions.  When the boundary conditions are
time independent, the results 3) and 4) are proved in \cite{MiSh1}.
The results 3), 4), and the existence part of 2) for time dependent
boundary conditions are new.  The strong monotonicity result 1)
basically follows from \cite[Theorem 11.6]{Ama3} and strongly maximum
principal and the Hopf boundary point principle for classical
solutions of parabolic equations.

\medskip
We now consider the averaged equations  of \eqref{nonauton-eq1} and
\eqref{random-eq1} in the following sense:

In the case of~\eqref{nonauton-eq1} we call $(\hat{c}(\cdot),
\hat{d}(\cdot))$ an {\em averaged function\/} of $(c,d)$ if
\begin{equation*}
\hat{c}(x) = \lim_{n\to\infty}\frac{1}{T_n-S_n}
\int_{S_n}^{T_n}c(t,x)\,dt \quad\text{for } x \in D
\end{equation*}
and
\begin{equation*}
\hat{d}(x) =
\lim_{n\to\infty}\frac{1}{T_n-S_n}\int_{S_n}^{T_n}d(t,x)\,dt
\quad\text{for } x \in \p D
\end{equation*}
for some $T_n - S_n \to \infty$, where the limit is uniform in $x \in
\bar{D}$ (resp. in $x \in \p D$).

In the case of~\eqref{random-eq1} we call $(\hat{c}(\cdot),
\hat{d}(\cdot))$ the {\em averaged function\/} of $(c,d)$ if
\begin{equation*}
\hat{c}(x) = \int_\Omega c(\omega,x)\,d\PP(\omega) \quad \text{for }
x \in D
\end{equation*}
and
\begin{equation*}
\hat{d}(x) = \int_\Omega d(\omega,x)\,d\PP(\omega) \quad \text{for }
x \in \p D.
\end{equation*}

The equation
\begin{equation}
\label{nonauton-or-random-avg}
\begin{cases}
\disp\frac{\p u}{\p t} = \sum_{i,j=1}^{N} a_{ij}(x)\frac{\p^{2}u}{\p
x_{i} \p x_{j}} + \sum_{i=1}^{N}
a_i(x)\frac{\p u}{\p x_i} + \hat{c}(x)u, \quad & x \in D, \\[1.5ex]
\hat{\mathcal{B}}u = 0, \quad & x \in \p D,
\end{cases}
\end{equation}
where
\begin{equation*}
\hat{\mathcal{B}}u =
\begin{cases}
u & \text{(Dirichlet)}
\\[1.5ex]
\disp \sum_{i=1}^{N} b_i(x)\frac{\p u}{\p x_i} & \text{(Neumann)}\\[1.5ex]
\disp \sum_{i=1}^{N} b_i(x)\frac{\p u}{\p x_i}+ \hat{d}(x)u &
\text{(Robin)},
\end{cases}
\end{equation*}
is called {\em an averaged equation\/} of~\eqref{nonauton-eq1} ({\em
the averaged equation\/} of~\eqref{random-eq1}) if
$(\hat{c},\hat{d})$ is an averaged function of $(c,d)$ (the averaged
function of $(c,d)$).

Denote $\lambda(\hat{c},\hat{d})$ to be the principal eigenvalue of
\eqref{nonauton-or-random-avg}.  We then have the following main
results of the paper.

\medskip
{\em
\begin{itemize}
\item[5)]
Consider \eqref{nonauton-eq1}.  Then $\lambdainf(c,d) \ge
\lambda(\hat{c},\hat{d})$ for some averaged function
$(\hat{c},\hat{d})$ of $(c,d)$ and $\lambdasup(c,d) \ge
\lambda(\hat{c},\hat{d})$ for any averaged function
$(\hat{c},\hat{d})$ of $(c,d)$ \textup{(}see Theorem
\ref{ch5-smoothlb-thm1}(1)\textup{)}.  Moreover, if $(c,d)$ is
uniquely ergodic and minimal, then $\lambdainf(c,d)( =
\lambdasup(c,d)) = \lambda(\hat{c},\hat{d})$ for the
\textup{(}necessarily unique\textup{)} averaged function
$(\hat{c},\hat{d})$ of $(c,d)$ if and only if $c(t,x) = c_{1}(x)
+ c_{2}(t)$ and $d(t,x) = d(x)$ \textup{(}see
Theorem~\ref{ch5-smoothlb-thm2}(1)\textup{)}.
\item[6)]
Consider \eqref{random-eq1}.  Then $\lambda(c,d) \ge
\lambda(\hat{c},\hat{d})$ \textup{(}see Theorem
\ref{ch5-smoothlb-thm1}(2)\textup{)}.  Further, $\lambda(c,d) =
\lambda(\hat{c},\hat{d})$ if and only if there is $\Omega^{*}
\subset \Omega$ with $\PP(\Omega^{*}) = 1$ such that
$c(\theta_{t}\omega,x) = c_{1}(x) + c_{2}(\theta_{t}\omega)$ for
any $\omega \in \Omega^{*}$, $t \in \RR$ and $x \in \bar{D}$, and
$d(\theta_{t}\omega,x) = d(x)$ for any $\omega \in \Omega^{*}$,
$t \in \RR$ and $x \in \p D$ \textup{(}see
Theorem~\ref{ch5-smoothlb-thm2}(2)\textup{)}.
\end{itemize}
}

\medskip
\noindent Hence time variations cannot reduce the principal spectrum
and principal Lyapunov exponent (or the principal eigenvalues of the
time averaged equations give lower bounds of principal spectrum and
principal Lyapunov exponent of non-averaged equations).  Indeed, the
time variations increase the principal spectrum and principal
Lyapunov exponents except in the degenerate cases.  In the biological
context these results mean that invasion by a new species (see
\cite{CC}, p.~220) is always easier in the time-dependent case or
that time variations favor persistence (viewing both
\eqref{nonauton-eq1} and \eqref{nonauton-or-random-avg} as linear
population growth models, then by 5), positive solutions of all
averaged equations \eqref{nonauton-or-random-avg} of
\eqref{nonauton-eq1} bounded away from zero implies positive
solutions of \eqref{nonauton-eq1} also bounded away from zero, but
not vice~versa in general).

It should be pointed out that the results 5), 6) have been proved in
\cite{HuShVi} and \cite{MiSh1} when the boundary conditions are time
independent.  They are new when the boundary conditions are time
dependent and the proof presented in this paper is not the same as
those in \cite{HuShVi} and \cite{MiSh1}.

It should be also pointed out that the results 1)--4) apply to fully
time dependent/random parabolic equations (i.e., equations in which
all the coefficients can depend on $t$/$\theta_t\omega$).  But 5) and
6) are mainly for equations of form \eqref{nonauton-eq1} and
\eqref{random-eq1}, respectively.

\medskip

The rest of the paper is organized as follows.  In
Section~\ref{preliminaries} we collect several elementary lemmas and
introduce some standing notations for future reference.  We review
some existence and regularity theorems  and construct the
skew-product semiflow generated by \eqref{nonauton-eq1} and
\eqref{random-eq1} in Section~\ref{skew-product}.
Section~\ref{strong-monotonicity} is devoted to the study of the
monotonicity of the skew-product semiflow constructed in
Section~\ref{skew-product} and the existence of global positive
solutions of \eqref{nonauton-eq2}. Definition and basic properties of
principal spectrum and principal Lyapunov exponents are discussed in
Section~\ref{principal-spectrum}.  We prove the time averaging
results in Section~\ref{main-results}.

\medskip
The authors are grateful to the referees for their remarks.

\section{Elementary Lemmas and notations}
\label{preliminaries}
We collect first, for further reference, some elementary results.

First of all, let $Z$ be a compact metric space and ${\mathcal B}(Z)$
be the Borel $\sigma$-algebra of $Z$.  $(Z,\RR) :=
(Z,\{\sigma_t\}_{t\in\RR})$ is called a {\em compact flow\/} if
$\sigma_t \colon Z \to Z$ ($t \in \RR$) satisfies:  $[\, (t,z)
\mapsto \sigma_{t}z\,]$ is jointly continuous in $(t,z) \in \RR
\times Z$, $\sigma_0 = \mathrm{id}$, and $\sigma_s \circ \sigma_t =
\sigma_{s+t}$ for any $s, t \in \RR$. We may write $z \cdot t$ or
$(z,t)$ for $\sigma_{t}z$.  A probability measure $\mu$ on $(Z,
\mathcal{B}(Z))$ is called an {\em invariant measure\/} for
$(Z,\{\sigma_t\}_{t\in\RR})$ if for any $E \in \mathcal{B}(Z)$ and
any $t \in \RR$, $\mu(\sigma_t(E)) = \mu(E)$. An invariant measure
$\mu$ for $(Z,\{\sigma_t\}_{t\in\RR})$ is said to be {\em ergodic\/}
if for any $E \in \mathcal{B}(Z)$ satisfying $\mu(\sigma_t^{-1}(E)
\bigtriangleup E) = 0$ for all $t \in \RR$, $\mu(E) = 1$ or $\mu(E) =
0$. The compact flow $(Z,\{\sigma_t\}_{t\in\RR})$ is said to be {\em
uniquely ergodic\/} if it has a unique invariant measure (in such
case, the unique invariant measure is necessarily ergodic).  We say
that $(Z,\{\sigma_t\}_{t\in\RR})$ is {\em minimal\/} or {\em
recurrent\/} if for any $z \in Z$, the orbit $\{\,\sigma_{t}z:
t\in\RR\,\}$ is dense in $Z$.

Let $\OFP$ be a probability space, $\{\theta_t\}_{t\in\RR}$ be a
family of $\PP$-preserving transformations (i.e.,
$\PP(\theta_t^{-1}(F)) = \PP(F)$ for any $F \in \mathcal{F}$ and $t
\in \RR$) such that $(t,\omega) \mapsto \theta_{t}\omega$ is
measurable, $\theta_0 = \mathrm{id}$, and $\theta_{t+s} = \theta_{t}
\circ \theta_{s}$ for all $t, s \in \RR$.  Thus
$\{\theta_t\}_{t\in\RR}$ is a flow on $\Omega$ and
$(\OFP,\{\theta_t\}_{t\in\RR})$ is called a {\em metric dynamical
system\/}.  $(\OFP,\{\theta_t\}_{t\in\RR})$ is said to be {\em
ergodic\/} if for any $F \in \mathcal{F}$ satisfying
$\PP(\theta_t^{-1}(F) \bigtriangleup F) = 0$ for any $t \in \RR$,
$\PP(F) = 1$ or $\PP(F) = 0$.

In the following, we assume that $(\OFP,\{\theta_t\}_{t\in\RR})$ is
an ergodic metric dynamical system.
\begin{lemma}
\label{ch5-pre-holder-lm}
\begin{itemize}
\item[{\rm (1)}]
Let $h_i \colon [0,T] \times D \to \RR$ \textup{(}$i = 1,2,\dots,
N$\textup{)} be square-integrable in $t \in [0,T]$ and $a_{ij} =
a_{ji} \colon D \to \RR$ \textup{(}$i, j = 1,2,\dots,N$\textup{)}
satisfy
\begin{equation*}
\sum_{i,j=1}^N a_{ij}(x)\xi_i\xi_j \ge \alpha_0 \sum_{i=1}^N \xi_i^2
\end{equation*}
for some $\alpha_0 > 0$ and any $x \in \bar{D}$, $\xi =
(\xi_1,\xi_2,\dots,\xi_N)^{\top} \in \RR^N$.  Then for any $x \in
D$,
\begin{align*}
& \sum_{i,j=1}^N a_{ij}(x)\; \frac {1}{T}\int_0^T h_i(t,x)\,dt \;
\frac{1}{T}\int_0^T h_j(t,x)\,dt \\
\le & \sum_{i,j=1}^N a_{ij}(x) \; \frac{1}{T}\int_0^T
h_i(t,x)h_j(t,x)\,dt.
\end{align*}
Moreover, the equality holds at some $x_0 \in D$ if and only if
$h_i(t,x_0) = \tilde{h}_i(x_0)$ for some $\tilde{h}_i(x_0)$
\textup{(}$i = 1,2,\dots,N$\textup{)} and a.e.\ $t \in [0,T]$.
\item[{\rm (2)}]
Let $h_i \colon \Omega \times D \to \RR$ \textup{(}$i =
1,2,\dots,N$\textup{)} be square-integrable in $\omega \in
\Omega$ and $a_{ij} = a_{ji} \colon D \to \RR$ $(i,j =
1,2,\dots,N)$ satisfy
\begin{equation*}
\sum_{i,j=1}^Na_{ij}(x)\xi_i\xi_j \ge \alpha_0 \sum_{i=1}^N\xi_i^2
\end{equation*}
for some $\alpha_0 > 0$ and any $x \in \bar{D}$, $\xi =
(\xi_1,\xi_2,\dots,\xi_N)^\top \in \RR^N$. Then for any $x \in
D$,
\begin{align*}
& \sum_{i,j=1}^N a_{ij}(x)\int_{\Omega}
h_i(\omega,x)\,d\PP(\omega)\int_\Omega h_j(\omega,x)\,d\PP(\omega)
\\
\le & \sum_{i,j=1}^{N} a_{ij}(x)\int_\Omega h_i(\omega,x)
h_j(\omega,x)\,d\PP(\omega).
\end{align*}
Moreover, the equality holds at some $x_0 \in D$ if and only if
$h_i(\omega,x_0) = \tilde{h}_i(x_0)$ for some $\tilde{h}_i(x_0)$
\textup{(}$i = 1,2,\dots, N$\textup{)} and a.e.\ $\omega \in
\Omega$.
\end{itemize}
\end{lemma}
\begin{proof}
See~\cite[Lemma~2.2]{HuShVi} for (1) and \cite[Lemma~3.5]{MiSh1}~for
(2).
\end{proof}

\begin{lemma}[Birkhoff's Ergodic Theorem]
\label{ch5-pre-ergodic-lm}
Let $h \in L^1\OFP$. Then there is an invariant measurable set
$\Omega_0 \subset \Omega$ such that $\PP(\Omega_0) = 1$ and
\begin{equation*}
\lim\limits_{T\to\infty} \frac{1}{T}\int\limits_0^{T}
h(\theta_t\omega)\,dt = \int\limits_\Omega h(\cdot)\,d\PP(\cdot)
\end{equation*}
for any $\omega \in \Omega_0$.
\end{lemma}
\begin{proof}
See \cite{Arn} or references therein.
\end{proof}

\begin{lemma}
\label{averaging-uniform}
Assume that $h \colon \Omega \times D \to \RR$ \textup{(}resp.~$h
\colon \Omega \times \bar{D} \to \RR$\textup{)} has the following
properties:
\begin{enumerate}
\item[{\rm (i)}]
$h(\cdot,x)$ belongs to $L^{1}(\Omega)$, for each $x \in D$,
\item[{\rm (ii)}]
for each $x \in D$ \textup{(}resp.~$x \in \bar{D}$\textup{)} and
each $\epsilon > 0$ there is $\delta > 0$ such that if $y \in D$
\textup{(}resp.~$y \in \bar{D}$\textup{)}, $\omega \in \Omega$
and $\abs{x - y} < \delta$ then $\abs{h(\omega,x) - h(\omega,y)}
< \epsilon$, where $\abs{\cdot}$ stands for the norm in $\RR^{N}$
or the absolute value, depending on the context.
\end{enumerate}
Denote, for each $x \in D$ \textup{(}resp.~$x \in \bar{D}$\textup{)},
\begin{equation*}
\hat{h}(x) := \int\limits_{\Omega} h(\omega,x)\, d\PP(\omega).
\end{equation*}
Then
\begin{enumerate}
\item[{\rm (a)}]
for any $x \in D$ \textup{(}resp.~$x \in \bar{D}$\textup{)} and
any $\epsilon > 0$ there is $\delta
> 0$ \textup{(}the same as in \textup{(ii))} such that if $y \in
D$ \textup{(}resp.~$y \in \bar{D}$\textup{)}, $\omega \in \Omega$
and $\abs{x - y} < \delta$ then $\abs{\hat{h}(x) - \hat{h}(y)} <
\epsilon$,
\item[{\rm (b)}]
there is a measurable $\Omega' \subset \Omega$ with $\PP(\Omega')
= 1$ such that
\begin{equation*}
\lim\limits_{T\to\infty} \frac{1}{T} \int\limits_{0}^{T}
h(\theta_{t}\omega,x)\, dt = \hat{h}(x)
\end{equation*}
for all $\omega \in \Omega'$ and all $x \in D$ \textup{(}resp.~$x
\in \bar{D}$\textup{)}. Moreover the convergence is uniform in $x
\in D_0$, for any compact $D_0 \Subset D$ \textup{(}resp. uniform
in $x \in \bar{D}$\textup{)}.
\end{enumerate}
\end{lemma}
\begin{proof}
Part (a) follows easily by the fact that the continuity is uniform in
$\omega \in \Omega$.  To prove (b), take a countable dense set
$\{x_l\}_{l=1}^{\infty}$ in $D$.  By Birkhoff's Ergodic Theorem
(Lemma~\ref{ch5-pre-ergodic-lm}) for each $l \in \NN$ there is a
measurable $\Omega_l \subset \Omega$ with $\PP(\Omega_l) = 1$ such
that
\begin{equation*}
\lim\limits_{T\to\infty} \frac{1}{T} \int\limits_{0}^{T}
h(\theta_{t}\omega,x_l)\, dt = \hat{h}(x_l)
\end{equation*}
for each $\omega \in \Omega_l$.  Take $\Omega' :=
\bigcap_{l=1}^{\infty} \Omega_l$.

Fix $x \in D$ (resp.~$x \in \bar{D}$).  For $\epsilon > 0$ take
$\delta > 0$ such that if $\abs{x - y} < \delta$ then
$\abs{h(\omega,x) - h(\omega,y)} < \epsilon/3$ and $\abs{\hat{h}(x) -
\hat{h}(y)} < \epsilon/3$.  Let $x_l$ be such that $\abs{x - x_l} <
\delta$, and let $T_0 > 0$ be such that
\begin{equation*}
\left\lvert \frac{1}{T} \int\limits_{0}^{T}
h(\theta_{t}\omega,x_l)\, dt - \hat{h}(x_l) \right\rvert <
\frac{\epsilon}{3}
\end{equation*}
for all $T > T_0$.  Then
\begin{equation*}
\left\lvert \frac{1}{T} \int\limits_{0}^{T} h(\theta_{t}\omega,x)\,
dt - \hat{h}(x) \right\rvert < \epsilon
\end{equation*}
for all $T > T_0$. (b) then follows.
\end{proof}

\begin{lemma}
\label{averaging-derivative}
Assume that $h \colon \Omega \times D \to \RR$ \textup{(}resp.~$h
\colon \Omega \times \bar{D} \to \RR$\textup{)} has the following
properties:
\begin{enumerate}
\item[{\rm (i)}]
$h(\cdot,x)$ belongs to $L^{1}(\Omega)$, for each $x \in D$,
\item[{\rm (ii)}]
$(\p h/\p x_i)(\omega,x)$ exists for each $\omega \in \Omega$ and
each $x \in D$ \textup{(}resp.~each $x \in \bar{D}$\textup{)};
further, $(\p h/\p x_i)(\cdot,x)$ belongs to $L^{1}(\Omega)$, for
each $x \in D$,
\item[{\rm (iii)}]
there exists $\alpha \in (0,1]$ such that for each $x \in D$
\textup{(}resp.~each $x \in \bar{D}$\textup{)} there are $L > 0$
and $\delta_0 > 0$ with the property that
\begin{equation*}
\left\lvert \frac{\p h}{\p x_i}(\omega,x) - \frac{\p h}{\p
x_i}(\omega,y) \right\rvert \le L \abs{x - y}^{\alpha}
\end{equation*}
for any $\omega \in \Omega$ and any $y \in D$ \textup{(}resp.~any
$y \in \bar{D}$\textup{)} with $|x - y| < \delta_0$.
\end{enumerate}
Denote, for each $x \in D$ \textup{(}resp.~$x \in \bar{D}$\textup{)},
\begin{equation*}
\hat{h}(x) := \int\limits_{\Omega} h(\omega,x)\, d\PP(\omega).
\end{equation*}
Then
\begin{enumerate}
\item[{\rm (a)}]
for each $x \in D$ \textup{(}resp.~each $x \in \bar{D}$\textup{)}
the derivative $(\p \hat{h}/\p x_i)(x)$ exists, and the equality
\begin{equation*}
\frac{\p \hat{h}}{\p x_i}(x) = \int\limits_{\Omega} \frac{\p h}{\p
x_i}(\omega,x)\, d\PP(\omega)
\end{equation*}
holds,
\item[{\rm (b)}]
for each $x \in D$ \textup{(}resp.~each $x \in \bar{D}$\textup{)}
there are $L > 0$ and $\delta_0 > 0$ \textup{(}the same as in
\textup{(}ii\textup{))} with the property that
\begin{equation*}
\left\lvert \frac{\p \hat{h}}{\p x_i}(x) - \frac{\p \hat{h}}{\p
x_i}(y) \right\rvert \le L \abs{x - y}^{\alpha}
\end{equation*}
for any $y \in D$ \textup{(}resp.~any $y \in \bar{D}$\textup{)}
with $\abs{x - y} < \delta_0$,
\item[{\rm (c)}]
there is a measurable $\Omega' \subset \Omega$ with $\PP(\Omega')
= 1$ such that
\begin{equation*}
\frac{\p \hat{h}}{\p x_i}(x) = \lim\limits_{T\to\infty} \frac{1}{T}
\int\limits_{0}^{T} \frac{\p h}{\p x_i}(\theta_{t}\omega,x)\, dt
\end{equation*}
for all $\omega \in \Omega'$ and all $x \in D$ \textup{(}resp.~$x
\in \bar{D}$\textup{)}. Moreover, the convergence is uniform in
$x \in D_{0}$, for any compact $D_0 \Subset D$ \textup{(}resp.
uniform in $x \in \bar{D}$\textup{)}.
\end{enumerate}
\end{lemma}
\begin{proof}
Parts (a) and (b) follow in a standard way.  Part (c) follows by an
application of Lemma~\ref{averaging-uniform}(b) to the function $(\p
h/\p x_i)(\omega,x)$.
\end{proof}

\medskip
From now on we assume that (A1)--(A6) are satisfied.

Consider the space $H$ consisting of $(\tilde{c}(\cdot,\cdot),
\tilde{d}(\cdot,\cdot))$, where $\tilde{c} \colon \RR \times \bar{D}
\to \RR$ and $\tilde{d} \colon \RR \times \p D \to \RR$ are bounded
continuous.  The set $H$ endowed with the topology of uniform
convergence on compact sets (the open-compact topology) becomes a
Fr\'echet space.

For $(\tilde{c}, \tilde{d}) \in H$ and $t \in \RR$ we define the {\em
time-translate\/} as $(\tilde{c}, \tilde{d}) \cdot t := ((\tilde{c}
\cdot t)(\cdot,\cdot), (\tilde{d} \cdot t)(\cdot,\cdot))$, where
$(\tilde{c} \cdot t)(s,x) := \tilde{c}(s+t,x)$, $s \in \RR$, $x \in
\bar{D}$, and $(\tilde{d} \cdot t)(s,x) := \tilde{d}(s+t,x)$, $s \in
\RR$, $x \in \p D$.  It is well~known that $(\tilde{c},\tilde{d})
\cdot t \in H$ whenever $(\tilde{c}, \tilde{d}) \in H$ and $t \in
\RR$, and that the mapping $[\, \RR \times H \ni
(t,(\tilde{c},\tilde{d})) \mapsto (\tilde{c},\tilde{d}) \cdot t \in H
\,]$ is continuous.

In the case of~\eqref{nonauton-eq1} let $Y = Y(c,d) := \cl\{\,(c,d)
\cdot t:t \in \RR\,\}$.  In the case of~\eqref{random-eq1} let $Y =
Y(\Omega) := \cl\{(c^{\omega},d^{\omega}):\omega \in \Omega\}$ (see
Section~\ref{introduction} for detail).

The following result is a consequence of the Ascoli--Arzel\`a
theorem.
\begin{lemma}
\label{lemma:properties-of-Y}
\begin{itemize}
\item[{\rm (i)}]
$Y$ is a compact subset of $H$.
\item[{\rm (ii)}]
For any $(\tilde{c},\tilde{d}) \in Y$ and any $t \in \RR$ there
holds $(\tilde{c},\tilde{d}) \cdot t \in Y$.
\item[{\rm (iii)}]
For any $(\tilde{c},\tilde{d}) \in Y$, $\tilde{c} \in
C^{2+\alpha,1+\alpha}(\RR \times \bar{D})$.  Moreover, the
$C^{2+\alpha,1+\alpha}(\RR \times \bar{D})$-norms are bounded
uniformly in $Y$ by the same bound as in \textup{(A3)}.
\item[{\rm (iv)}]
For any $(\tilde{c},\tilde{d}) \in Y$, $\tilde{d} \in
C^{2+\alpha,3+\alpha}(\RR \times \p D)$. Moreover, the
$C^{2+\alpha,3+\alpha}(\RR \times \p D)$-norms are bounded
uniformly in $Y$ by the same bound as in \textup{(A4)}.
\item[{\rm (v)}]
For a sequence $(\tilde{c}^{(n)},\tilde{d}^{(n)}) \to
(\tilde{c},\tilde{d})$ in $Y$, the mixed derivatives of
$\tilde{c}^{(n)}$ of order up~to $2$ in $t$ and up~to $1$ in $x$
converge to the respective derivatives of $\tilde{c}$, uniformly
on compact subsets of $\RR \times \bar{D}$.
\item[{\rm (vi)}]
For a sequence $(\tilde{c}^{(n)},\tilde{d}^{(n)}) \to
(\tilde{c},\tilde{d})$ in $Y$, the mixed derivatives of
$\tilde{d}^{(n)}$ of order up~to $2$ in $t$ and up~to $3$ in $x$
converge to the respective derivatives of $\tilde{d}$, uniformly
on compact subsets of $\RR \times \p D$.
\end{itemize}
\end{lemma}

We write $\sigma_{t}(\tilde{c},\tilde{d})$ for $(\tilde{c},\tilde{d})
\cdot t \in Y$.  We will denote by $(Y,\RR)$ the compact flow $(Y,
\{\sigma_{t}\}_{t\in\RR})$.

Consider \eqref{nonauton-eq1}.  For $x \in \bar{D}$ and $S < T$ we
denote
\begin{equation*}
\bar{c}(x;S,T) := \frac{1}{T-S}\int\limits_{S}^{T} c(t,x) \, dt.
\end{equation*}
Similarly, for $x \in \p D$ and $S < T$ we denote
\begin{equation*}
\bar{d}(x;S,T) := \frac{1}{T-S}\int\limits_{S}^{T} d(t,x) \, dt.
\end{equation*}
Let
\begin{align}
\label{hat-Y-eq}
\hat{Y}(c,d) := \{\,(\hat{c},\hat{d}):\ & \exists\, S_n < T_n \
\text{with} \  T_n-S_n \to \infty
\ \text{such that} \nonumber \\
& \  (\hat{c},\hat{d}) = \lim_{n\to\infty}
(\bar{c}(\cdot;S_n,T_n),\bar{d}(\cdot;S_n,T_n))\,\},
\end{align}
where the convergence is in $C(\bar{D}) \times C(\p D)$.

The following result is a consequence of the Ascoli--Arzel\`a theorem
(compare Lemma \ref{lemma:properties-of-Y}).
\begin{lemma}
\label{lemma:properties-of-Y(hat-a)}
\begin{itemize}
\item[{\rm (i)}]
$\hat{Y}(c,d)$ is a nonempty compact subset of $C(\bar{D}) \times
C(\p D)$.
\item[{\rm (ii)}]
For any $(\hat{c},\hat{d}) \in \hat{Y}(c,d)$, $\hat{c} \in
C^{1}(\bar{D})$.  Moreover, the $C^{1}(\bar{D})$-norms are
bounded uniformly in $\hat{Y}(c,d)$.
\item[{\rm (iii)}]
For any $(\hat{c},\hat{d}) \in \hat{Y}(c,d)$, $\hat{d} \in
C^{3}(\p D)$. Moreover, the $C^{3}(\p D)$-norms are bounded
uniformly in $\hat{Y}(c,d)$.
\end{itemize}
\end{lemma}

\vskip2ex
\begin{definition}
\label{ch5-pre-def}
\begin{itemize}
\item[{\rm (1)}]
Let $a$ be as in \eqref{nonauton-eq1}.  We say $(c,d)$ is {\em
uniquely ergodic} if the compact flow $(Y(c,d),\RR)$ is uniquely
ergodic.
\item[{\rm (2)}]
Let $a$ be as in \eqref{nonauton-eq1}.  We say $(c,d)$ is {\em
minimal} or {\em recurrent} if $(Y(c,d),\RR)$ is minimal.
\end{itemize}
\end{definition}

\begin{remark}
\begin{itemize}
\item[(1)]
If $c(t,x)$ and $d(t,x)$ are almost periodic in $t$ uniformly
with respect to $x \in \bar{D}$ and $x \in \p D$, respectively,
then $(c,d)$ is both uniquely ergodic and minimal.

\item[(2)]
If $c(t,x)$ and $d(t,x)$ are almost automorphic in $t$ uniformly
with respect to $x \in \bar{D}$ and $x \in \p D$, respectively,
then $(c,d)$ is minimal, but it may not be uniquely ergodic
\textup{(}see \cite{Jon} for examples\textup{)}.

\item[(3)]
There is $(c,d)$ which is neither uniquely ergodic nor minimal.
For example, let $c(t,x) = \tan^{-1}(t)$  and $d(t,x) \equiv 1$,
then $\{(\pi/2,1)\}$ and $\{(-\pi/2,1)\}$ are two minimal
invariant subsets of $Y(c,d)$, and hence $Y(c,d)$ is neither
uniquely ergodic nor minimal.
\end{itemize}
\end{remark}

\begin{lemma}
\label{ch5-pre-existence-avg-lm}
Consider \eqref{nonauton-eq1} with $(c,d)$ uniquely ergodic, $\mu$
being the unique ergodic measure.  For $(\tilde{c},\tilde{d}) \in
Y(c,d)$ put $\tilde{c}_0(x) := \tilde{c}(0,x)$ and $\tilde{d}_0(x) :=
\tilde{d}(0,x)$.  Then
\begin{equation}
\label{time-space-avg-eq1}
\lim\limits_{T\to\infty} \frac{1}{T}\int\limits_0^T c(t,x)\,dt =
\int\limits_{Y(c,d)} \tilde{c}_0(x) \, d\mu((\tilde{c},\tilde{d}))
\end{equation}
uniformly for $x \in D$, and
\begin{equation}
\label{time-space-avg-eq2} \lim\limits_{T\to\infty}
\frac{1}{T}\int_0^T d(t,x)\,dt = \int\limits_{Y(c,d)} \tilde{d}_0(x)
\, d\mu((\tilde{c},\tilde{d}))
\end{equation}
uniformly for  $x \in \p D$.
\end{lemma}
\begin{proof}
We prove only \eqref{time-space-avg-eq1}, the other proof being
similar.  It follows via the Ascoli--Arzel\`a theorem that the set
$\{\, (1/T)\int_{0}^{T} c(t,\cdot)\,dt : T > 0\,\} =
\{\,\bar{c}(\cdot;0,T) : T > 0 \,\}$ has compact closure in
$C(\bar{D})$, consequently from any sequence $(T_n)$ with
$\lim_{n\to\infty}T_n = \infty$ one can extract a subsequence
$(T_{n_k})$ such that $\bar{c}(\cdot;0,T_{n_k})$ converges uniformly
in $x \in \bar{D}$ to some $\check{c}$ (depending perhaps on the
subsequence).

On the other hand, as $(Y(c,d),\RR)$ is uniquely ergodic, for each
continuous $g \colon Y(c,d) \to \RR$ there holds
\begin{equation*}
\lim\limits_{T\to\infty} \frac{1}{T} \int\limits_{0}^{T}g((c,d)
\cdot t)\,dt = \int\limits_{Y(c,d)} g(\cdot) \, d\mu(\cdot)
\end{equation*}
(compare, e.g.,\ Oxtoby~\cite{Oxt}).  Fix $x \in \bar{D}$ and take
$g((\tilde{c},\tilde{d})) := \tilde{c}_0(x)$.  We have thus obtained
that if $\bar{c}(x;0,T_n)$ converges, for some $T_n \to \infty$,
uniformly in $x \in \bar{D}$, then the limit is always equal to
$\check{c}(x) = \int_{Y(c,d)} \tilde{c}_0(x) \, d\mu$.
\end{proof}

We introduce the following standing notations ($X_1$, $X_2$ are
Banach spaces):

\medskip

$\mathcal{L}(X_1,X_2)$ represents the space of all bounded linear
operators from $X_1$ to $X_2$, endowed with the norm topology;

$\lVert \cdot \rVert_{X_1}$ denotes the norm in $X_1$;

$X_1^{*}$ denotes the Banach space dual to $X_1$;

$(\cdot,\cdot)_{X_1,X_1^{*}}$ stands for the duality pairing between
$X_1$ and $X_1^{*}$;

$\norm{\cdot}$ denotes the norm in $L_2(D)$ or the norm in
$\mathcal{L}(L_2(D),L_2(D))$;

$\langle \cdot, \cdot \rangle$ stands for the standard inner product
in $L_2(D)$;

$\lVert \cdot \rVert_{X_1,X_2}$ indicates the norm in
$\mathcal{L}(X_1,X_2)$;

$[\cdot,\cdot]_\theta$  is a complex interpolation functor;

$(\cdot,\cdot)_{\theta,p}$ is a real interpolation functor (see
\cite{BeLo}, \cite{Tri} for more detail);

$\ZZ$ denotes the set of integers;

$\NN$ denotes the set of nonnegative integers.

\section{Skew-product semiflows}
\label{skew-product}
We  construct in this section a linear skew-product semiflow on $X$
generated by~\eqref{nonauton-eq1} or by \eqref{random-eq1}, where $X$
is as in \eqref{X-eq}.

To do so, we first use the theory presented by H. Amann in~\cite{Ama}
to consider the existence  of solution of
\eqref{nonauton-eq2}+\eqref{initial-condition} for any $(\tilde{c},
\tilde{d}) \in Y$ and any $u_0 \in L_p(D)$.  Recall that we assume
(A1)--(A6) throughout.

Let $\mathcal{A}(\tilde{c})$ denote the operator given by
\begin{equation*}
\mathcal{A}(\tilde{c})u = \sum_{i,j=1}^{N} a_{ij}(x)
\frac{\p^{2}u}{\p x_{i} \p x_{j}} + \sum_{i=1}^{N} a_i(x) \frac{\p
u}{\p x_i} + \tilde{c}(0,x)u, \quad x \in D,
\end{equation*}
and let $\mathcal{B}(\tilde{d})$ denote the boundary operator given
by
\begin{equation*}
\mathcal{B}(\tilde{d})u =
\begin{cases}
u & x \in \p D \quad \text{(Dirichlet)}
\\[1.5ex]
\disp \sum_{i=1}^{N} b_i(x) \frac{\p u}{\p x_i} & x \in \p D \quad
\text{(Neumann)}\\[1.5ex]
\disp \sum_{i=1}^{N} b_i(x) \frac{\p u}{\p x_i} + \tilde{d}(0,x)u &
x \in \p D \quad  \text{(Robin)}.
\end{cases}
\end{equation*}
Let
\begin{equation*}
V_{p}^{1}(\tilde{d}) := \{\,u \in W_{p}^{2}(D):
\mathcal{B}(\tilde{d})u = 0\,\}.
\end{equation*}
For given $0 < \theta < 1$ and $1 < p < \infty$,  let
\begin{equation*}
\Vptheta :=
\begin{cases}
(L_p(D),W_{p}^{2}(D))_{\theta,p} \quad
\text{if }2\theta \not \in \NN \\[2ex]
[L_p(D),W_{p}^{2}(D)]_{\theta} \quad \text{if } 2\theta \in \NN
\end{cases}
\end{equation*}
and
\begin{equation*}
\Vptheta(\tilde{d}) :=
\begin{cases}
(L_p(D),V_{p}^{1}(\tilde{d}))_{\theta,p}
\quad \text{if } 2\theta \not \in \NN \\[2ex]
[L_p(D), V_{p}^{1}(\tilde{d})]_{\theta} \quad \text{if } 2\theta \in
\NN.
\end{cases}
\end{equation*}

\begin{proposition}
\label{X-theta-spaces-prop}
\begin{itemize}
\item[{\rm (1)}]
$\Vptheta = W_{p}^{2\theta}$.
\item[{\rm (2)}]
If $2\theta - \frac{1}{p} \ne 0, 1$ then $\Vptheta(\tilde{d})$ is
a closed subspace of $\Vptheta$.
\end{itemize}
\end{proposition}
\begin{proof} (1) follows from \cite[Theorem 11.6]{Ama}.

(2) follows from \cite[Lemma 14.4]{Ama}.
\end{proof}

Recall the following compact embedding:
\begin{equation}
\label{sobolev-embed-eq1}
W_{p}^{j+m}(D) \hookrightarrow C^{j,\lambda}(\bar{D})
\end{equation}
if  $mp > N > (m-1)p$ and $ 0 < \lambda < m - (N/p)$, and
\begin{equation}
\label{sobolev-embed-eq2}
W_{p}^{2}(D) \hookrightarrow \Vptheta
\end{equation}
\begin{equation}
\label{sobolev-embed-eq3}
V_{p}^{1}(\tilde{d}) \hookrightarrow \Vptheta(\tilde{d})
\end{equation}
for any  $0 \leq \theta < 1$ and $\tilde{a} \in Y$, where $V_p^0,
V_p^0(\tilde d)=L_p(D)$.

Let
\begin{equation*}
A_{(\tilde{c},\tilde{d}),p}(t) := \mathcal{A}(\tilde{c} \cdot
t)|_{V_p^1(\tilde{d} \cdot t)}.
\end{equation*}
Then \eqref{nonauton-eq2}+\eqref{initial-condition} can be written as
\begin{equation}
\label{evolution-eq1}
\begin{cases} u_t = A_{(\tilde{c},\tilde{d}),p}(t) u \\
u(0) = u_0.
\end{cases}
\end{equation}

\begin{definition}
\label{lp-solution-def}
$u = u(t,x)$ is called an {\em $L_p$-solution\/} of
\textup{\eqref{nonauton-eq2}+\eqref{initial-condition}} if it is a
solution of the evolution equation \eqref{evolution-eq1} in $L_p(D)$.
\end{definition}
\begin{definition}
\label{classical-solution-def} $u = u(t,x)$ defined on $(t_0,t_1)
\times \bar{D}$, $t_0 < t_1$, is a {\em classical solution\/} of
\eqref{nonauton-eq2} on $(t_0,t_1)$ if it is continuous on
$(t_0,t_2)\times \bar D$,  it satisfies the differential equation in
\eqref{nonauton-eq2} for all $t \in (t_0,t_1)$ and all $x \in D$, and
it satisfies the boundary conditions for all $t \in (t_0,t_1)$ and
all $x \in \p D$.
\end{definition}

The following existence result follows from~\cite[Theorem~15.1]{Ama}.

\begin{proposition}
\label{existence-prop} For each $(\tilde{c},\tilde{d}) \in Y$ and
each $u_0 \in L_{p}(D)$ there exists a unique $L_{p}(D)$-solution
$U_{(\tilde{c},\tilde{d}),p}(\cdot,0)u_0 \colon [0,\infty) \to
L_{p}(D)$ of~\textup{\eqref{nonauton-eq2}+\eqref{initial-condition}}.
\end{proposition}

It follows from the uniqueness of $L_p$-solutions that the following
{\em cocycle property\/} for the solution operator holds:
\begin{equation}
\label{cocycle0}
U_{(\tilde{c},\tilde{d}),p}(t+s,0) = U_{(\tilde{c},\tilde{d}) \cdot
s, p}(t,0) U_{(\tilde{c},\tilde{d}),p}(s,0) \qquad \text{for any }
(\tilde{c},\tilde{d}) \in Y, \ s, t \ge 0.
\end{equation}

We collect now the regularity properties of the $L_{p}(D)$-solutions
which will be useful in the sequel.
\begin{proposition}
\label{evolution-op-in-smooth-case-prop1}
For any $1 < p < \infty$, $(\tilde{c},\tilde{d}) \in Y$ and $u_0 \in
L_2(D)$ there holds $U_{(\tilde{c},\tilde{d}),2}(t,0)u_0 \in
V_{p}^{1}(\tilde{d} \cdot t)$ for $t > 0$.  Moreover, for any fixed
$0 < t_1 \le t_2$ there is $C_{p} = C_{p}(t_1,t_2) > 0$ such that
\begin{equation*}
\lVert U_{(\tilde{c},\tilde{d}),2}(t,0)
\rVert_{L_{2}(D),W_{p}^{2}(D)} \le C_{p}
\end{equation*}
for all $(\tilde{c},\tilde{d}) \in Y$ and $t_1 \le  t \le t_2$.
\end{proposition}
\begin{proof}
First of all, by  \cite[Lemma~6.1 and~Theorem~14.5]{Ama}, for any $1
< p < \infty$, any $(\tilde{c},\tilde{d}) \in Y$,  and any $u_0 \in
L_p(D)$,
\begin{equation}
\label{smoothness-of-solution1}
U_{(\tilde{c},\tilde{d}),p}(t,0)u_0 \in V_p^1(\tilde{d} \cdot
t)\quad \text{for } t > 0.
\end{equation}
Moreover, for any $t_2 > 0$, there is $C_p = C_p(t_2) > 0$ such that
\begin{equation}
\label{smoothness-of-solution2}
\lVert U_{(\tilde{c},\tilde{d}),p}(t,0)
\rVert_{L_{p}(D),W_{p}^{2}(D)} \le \frac{C_p}{t}
\end{equation}
for all $(\tilde{c},\tilde{d}) \in Y$ and $0 < t \le t_2 $.

Next, note that if $1 < p \le 2$, then we have $L_2(D) \subset
L_p(D)$, $V_2^1(\tilde{d} \cdot t)\subset V_p^1(\tilde{d} \cdot t)$,
and $W_2^2(D) \subset W_p^2(D)$.  The proposition then follows from
\eqref{smoothness-of-solution1} and \eqref{smoothness-of-solution2}.

Now, assume $2 < p < \infty$. If $4 \ge N$, then by Sobolev
embeddings (see \cite[Theorem 6.2]{Ada}, we have
\begin{equation}
\label{smoothness-of-solution3}
W_2^2(D)\hookrightarrow C(\bar{D}).
\end{equation}
Then it follows with the help of~\eqref{smoothness-of-solution2} that
$U_{(\tilde{c},\tilde{d}),2}(t/2,0)u_0 \in L_p(D)$ for all $t > 0$.
\eqref{smoothness-of-solution1} gives that
$U_{(\tilde{c},\tilde{d}),2}(t,0)u_0 \in V_p^1(\tilde{d} \cdot t)$
for all $t > 0$.  We estimate
\begin{align*}
& \lVert U_{(\tilde{c},\tilde{d}),2}(t,0)
\rVert_{L_{2}(D),W_{p}^{2}(D)} \\
 \le& \tilde{C} \lVert U_{(\tilde{c},\tilde{d}) \cdot
(t_1/2),p}(t-t_1/2,0) \rVert_{L_p(D),W_p^2(D)} \cdot \lVert
U_{(\tilde{c},\tilde{d}),2}(t_1/{2},0)u_0 \rVert_{L_2(D),W_2^2(D)} \\
 \le& \tilde{C} \cdot \frac{C_p(t_2-t_1/2)}{t-t_1/2} \cdot
\frac{C_2(t_1/2)}{t_1/2}
\end{align*}
for all $(\tilde{c},\tilde{d}) \in Y$ and $t_1 \le  t \le t_2$, where
$\tilde{C}$ denotes the norm of the embedding $W_2^2(D)
\hookrightarrow L_p(D)$.  Hence the proposition also holds.

Finally, assume $p > 2$ and $N > 4$. There are $l \in \NN$ and $p_0 =
2 < p_1 < p_2 < \dots < p_l$ such that $p_{i-1} < p_i <
\frac{Np_{i-1}}{N-2p_{i-1}}$ for $i = 1, 2, \dots, l$, and $2p_l
> N$.  For any $\delta > 0$, let $0 = \tau_0 < \tau_1 < \tau_2 < \dots
< \tau_l = \frac{\delta}{2}$.  By \eqref{smoothness-of-solution2},
\begin{equation*}
\lVert U_{(\tilde{c},\tilde{d}) \cdot \tau_i,p_i}
(\tau_{i+1}-\tau_i,0) \rVert_{L_{p_i}(D),W_{p_i}^2(D)} \le
\frac{C_{p_i}(\tau_{i+1}-\tau_i)}{\tau_{i+1}-\tau_i}
\end{equation*}
for $i = 0, 1, 2, \dots, l-1$.  By Sobolev embeddings (see
\cite[Theorem 6.2]{Ada}),
\begin{equation*}
W_{p_i}^2(D) \hookrightarrow L_{p_{i+1}}(D)
\end{equation*}
for $i = 0, 1, 2, \dots, l-1$.  We then have
\begin{equation}
\label{smoothness-aux}
\lVert U_{(\tilde{c},\tilde{d}) \cdot
\tau_{i-1},p_i}(\tau_{i+1}-\tau_i,0)
\rVert_{L_{p_i}(D),L_{p_{i+1}}(D)} \le \tilde{C}_{p_i}
\end{equation}
for some $\tilde{C}_{p_i} > 0$.  Further, since $p_l > \frac{N}{2}$,
by Sobolev embeddings (see \cite[Theorem 6.2]{Ada}),
\begin{equation*}
W_{p_l}^2(D) \hookrightarrow C(\bar{D}).
\end{equation*}
Consequently, we have an embedding $W_{p_l}^2(D) \hookrightarrow
L_p(D)$ (denote its norm by $\bar{C}$).  It then follows that for any
$u_0 \in L_2(D)$,
\begin{align*}
U_{(\tilde{c},\tilde{d}),2}(\tau_l,0)u_0 & =U_{(\tilde{c},\tilde{d})
\cdot \tau_1, p_1}(\tau_l-\tau_1,0)U_{(\tilde{c},\tilde{d}),2}(\tau_1,0)u_0 \\
& = U_{(\tilde{c},\tilde{d}) \cdot
\tau_{l-1},p_{l-1}}(\tau_l-\tau_{l-1},0) \,
U_{(\tilde{c},\tilde{d}) \cdot
\tau_{l-2},p_{l-2}}(\tau_{l-1}-\tau_{l-2},0) \\
& \quad \dots U_{(\tilde{c},\tilde{d}) \cdot
\tau_1,p_1}(\tau_2-\tau_1,0)
U_{(\tilde{c},\tilde{d}),2}(\tau_1,0)u_0 \\
& \in L_p(D).
\end{align*}
This implies, via~\eqref{smoothness-of-solution1}, that
\begin{equation*}
U_{(\tilde c,\tilde d),2}(t,0)u_0 \in V_p^1(\tilde{d} \cdot t)
\end{equation*}
for any $t \ge \delta$ (and hence for any $t > 0$, since $\delta > 0$
is arbitrary).  Now we take $\delta = t_1$.  It follows
from~\eqref{smoothness-aux} and \eqref{smoothness-of-solution2} that
\begin{equation*}
\lVert U_{(\tilde{c},\tilde{d}),2}(t,0)
\rVert_{L_{2}(D),W_{p}^{2}(D)} \le \bar{C} \tilde{C}_{p_0}  \dots
\tilde{C}_{p_{l-1}} \frac{C_{p}(t_2-t_1/2)}{t_1/2}
\end{equation*}
for all $(\tilde{c},\tilde{d}) \in Y$ and $t_1 \le t \le t_2$.
\end{proof}
\begin{proposition}
\label{evolution-op-in-smooth-case-prop2} Suppose that $2\theta -
1/p \notin \NN$.  Then for any $t \ge 0$ and $u_0 \in
\Vptheta(\tilde{d})$ there holds $U_{(\tilde{c},\tilde{d}),p}(t,0)u_0
\in \Vptheta(\tilde{d} \cdot t)$. Moreover, for any $T > 0$ there is
$C_{p,\theta} = C_{p,\theta}(T) > 0$ such that
\begin{equation*}
\normVptheta{U_{(\tilde{c},\tilde{d}),p}(t,0)u_0} \le C_{p,\theta}
\normVptheta{u_0}
\end{equation*}
for any $(\tilde{c},\tilde{d}) \in Y$, $0 \le t \le T$, and $u_0 \in
\Vptheta(\tilde{d})$.
\end{proposition}
\begin{proof}
See \cite[Theorems~7.1 and~14.5]{Ama}.
\end{proof}

\begin{proposition}
\label{classical-solution-prop} For any $u_0\in L_p(D)$,
$U_{(\tilde{c},\tilde{d}),p}(t,0)u_0$ is a classical solution of
\eqref{nonauton-eq2} on $(0,\infty)$.
\end{proposition}
\begin{proof}
It follows from Proposition \ref{evolution-op-in-smooth-case-prop1}
and \cite[Corollary~15.3]{Ama}).
\end{proof}
Proposition \ref{classical-solution-prop} allows us to write
$U_{(\tilde{c},\tilde{d})}(t,0)u_0$ ($t > 0$) instead of
$U_{(\tilde{c},\tilde{d}),p}(t,0)u_0$.  In case of~\eqref{random-eq1}
we write $U_{\omega}(t,0)$ instead~of
$U_{(c^{\omega},d^{\omega)}}(t,0)$.

\smallskip
\begin{definition}
\label{global-solution-def}
A {\em global solution\/} of~\eqref{nonauton-eq2} is a classical
solution of~\eqref{nonauton-eq2} on $(-\infty,\infty)$.
\end{definition}
Observe that $v = v(t,x)$ is a global solution of
\eqref{nonauton-eq2} if and only if
\begin{equation*}
U_{(\tilde{c},\tilde{d}) \cdot t}(s,0)v(t,\cdot) =
v(t+s,\cdot)\quad\text{for any }t \in \RR \text{ and any }s \ge 0.
\end{equation*}
\vskip2ex
From now on, we assume (A7).  For any sequence $(\tilde{c}^{(n)},
\tilde{d}^{(n)})_{n=1}^{\infty} \subset Y$, we write
$\lim_{n\to\infty}(\tilde{c}^{(n)}$, $\tilde{d}^{(n)}) =
(\tilde{c},\tilde{d})$ if $(\tilde{c}^{(n)},\tilde{d}^{(n)})$
converges to $(\tilde{c},\tilde{d})$ in $Y$ as $n \to \infty$ (here
the convergence is uniform in the space variable and uniform  on
compact sets in the time variable).  We then present various
continuous dependence propositions.

\begin{proposition}[Joint continuity]
\label{joint-continuity-in-X-theta}
 For any sequence
$((\tilde{c}^{(n)},\tilde{d}^{(n)}))_{n=1}^{\infty} \subset Y$, any
sequence $(t_n)_{n=1}^{\infty} \subset (0,\infty)$ and any sequence
$(u_n)_{n=1}^{\infty} \subset L_2(D)$, if
$\lim_{n\to\infty}(\tilde{c}^{(n)},\tilde{d}^{(n)})$ $=
(\tilde{c},\tilde{d})$, $\lim_{n\to\infty}t_n = t$, where $t > 0$,
and $\lim_{n\to\infty}u_n = u_0$ in $L_2(D)$, then the following
holds.
\begin{itemize}
\item[(1)]
$U_{(\tilde{c}^{(n)},\tilde{d}^{(n)})}(t_n$, $0)u_n$ converges in
$\Vptheta$ to $U_{(\tilde{c},\tilde{d})}(t,0)u_0$, where $0 \le
\theta < 1$ and $1 < p < \infty$ with $2\theta - 1/p \notin \NN$.
\item[(2)]
$U_{(\tilde{c}^{(n)},\tilde{d}^{(n)})}(t_n,0)u_n$ converges in
$C^1(\bar{D})$ to $U_{(\tilde{c},\tilde{d})}(t,0)u_0$.
\end{itemize}
\end{proposition}
\begin{proof}
(1) Proposition~\ref{evolution-op-in-smooth-case-prop1} and
Eq.~\eqref{sobolev-embed-eq2} imply that there is a subsequence
$(n_k)_{k=1}^{\infty}$ such that
$U_{(\tilde{c}^{(n_k)},\tilde{d}^{(n_k)})}(t_{n_k},0)u_{n_k}$
converges, as $k \to \infty$, in $V_{p}^{\theta}$ to some $u^{*}$.
Note that
\begin{equation*}
\norm{U_{(\tilde{c},\tilde{d})}(t_n,0)u_0 -
U_{(\tilde{c},\tilde{d})}(t,0)u_0} \to 0
\end{equation*}
and
\begin{equation*}
\norm{U_{(\tilde{c},\tilde{d})}(t_n,0)u_n -
U_{(\tilde{c},\tilde{d})}(t_n,0)u_0} \to 0.
\end{equation*}
By~(A7) we have that
\begin{equation*}
\norm{U_{(\tilde{c}_n,\tilde{d}_n)}(t_n,0)u_n -
U_{(\tilde{c},\tilde{d})}(t_n,0)u_n} \to 0
\end{equation*}
and hence
\begin{equation*}
\norm{U_{(\tilde{c}_n,\tilde{d}_n)}(t_n,0)u_n -
U_{(\tilde{c},\tilde{d})}(t,0)u_0} \to 0
\end{equation*}
as $n \to \infty$.  As $V_{p}^{\theta}$ embeds continuously in
$L_2(D)$, we  must have  $u^{*} = U_{(\tilde{c},\tilde{d})}(t,0)u_0$
and the sequence
$U_{(\tilde{c}^{(n)},\tilde{d}^{(n)})}(t_{n},0)u_{n}$ converges, as
$n \to \infty$, in $\Vptheta$, to
$U_{(\tilde{c},\tilde{d})}(t,0)u_0$.

(2) It follows by (1)  and Eq.~\eqref{sobolev-embed-eq1}.
\end{proof}

\begin{proposition}[Norm continuity]
\label{norm-continuity-prop-in-X-theta}
\begin{itemize}
\item[(1)]
Let $1 < p < \infty$ and $2\theta - 1/p \not\in \NN$.  The
mapping
\begin{equation*}
[\,Y \times (0,\infty) \ni ((\tilde{c},\tilde{d}),t) \mapsto
U_{(\tilde{c},\tilde{d})}(t,0) \in \mathcal{L}(L_2(D),\Vptheta)\,]
\end{equation*}
is continuous.
\item[(2)] The mapping
\begin{equation*}
[\,Y \times (0,\infty) \ni ((\tilde{c},\tilde{d}),t) \mapsto
U_{(\tilde{c},\tilde{d})}(t,0) \in
\mathcal{L}(L_2(D),C^1(\bar{D}))\,]
\end{equation*}
is continuous.  Moreover, for any $t > 0$ and any
$(\tilde{c},\tilde{d}) \in Y$ the linear operator
$U_{(\tilde{c},\tilde{d})}(t,0)$ is compact \textup{(}completely
continuous\textup{)}.
\end{itemize}
\end{proposition}
\begin{proof}
(1) Assume that $(\tilde{c}^{(n)},\tilde{d}^{(n)})$ converges to
$(\tilde{c},\tilde{d})$ in $Y$ and that $t_n$ converges to $t > 0$.
Suppose to the contrary that
\begin{equation*}
\lVert U_{(\tilde{c}^{(n)},\tilde{d}^{(n)})}(t_n,0) -
U_{(\tilde{c},\tilde{d})}(t,0) \rVert_{L_2(D),\Vptheta} \not\to 0
\end{equation*}
as $n \to \infty$.  Then there are $\epsilon_0 > 0$ and a sequence
$(u_n)_{n=1}^{\infty} \subset L_2(D)$ with $\norm{u_n} = 1$ such that
\begin{equation*}
\normVptheta{U_{(\tilde{c}^{(n)},\tilde{d}^{(n)})}(t_n,0)u_n -
U_{(\tilde{c},\tilde{d})}(t,0)u_n} \ge \epsilon_0
\end{equation*}
for all $n$.  By Proposition~\ref{evolution-op-in-smooth-case-prop1},
there are $u^*$, $u^{**} \in \Vptheta$ such that (after possibly
extracting a subsequence)
\begin{equation*}
U_{(\tilde{c}^{(n)},\tilde{d}^{(n)})}(t_n,0)u_{n} \to u^{*}
\end{equation*}
and
\begin{equation*}
U_{(\tilde{c},\tilde{d})}(t,0)u_{n} \to u^{**}
\end{equation*}
in $\Vptheta$, as $n \to \infty$. Without loss of generality, we may
assume that  there is $\tilde u^*\in V_p^\theta$ such that
\begin{equation*}
U_{(\tilde{c},\tilde{d})}(t/2,0)u_n \to \tilde{u}^*
\end{equation*}
in $V_p^\theta$ as $n \to \infty$.  Then by Proposition
\ref{joint-continuity-in-X-theta}, we have
\begin{align*}
&\norm{U_{(\tilde{c},\tilde{d})}(t_n,0)u_n -
U_{(\tilde{c},\tilde{d})}(t,0)u_n} \\
 =& \norm{U_{(\tilde{c},\tilde{d}) \cdot
{t/2}}(t_n-t/2,0)U_{(\tilde{c},\tilde{d})}(t/2,0)u_n -
U_{(\tilde{c},\tilde{d}) \cdot
{t/2}}(t/2,0)U_{(\tilde{c},\tilde{d})}(t/2,0)u_n} \\
 \to& \norm{U_{(\tilde{c},\tilde{d}) \cdot t/2}(t/2,0) \tilde{u}^* -
U_{(\tilde{c},\tilde{d}) \cdot t/2}(t/2,0) \tilde{u}^*} = 0
\end{align*}
as $n \to \infty$.  By the property (A7) we have
\begin{equation*}
\norm{U_{(\tilde{c}^{(n)},\tilde{d}^{(n)})}(t_n,0) -
U_{(\tilde{c},\tilde{d})}(t_n,0)} \to 0
\end{equation*}
as $n \to \infty$.  Then we must have $u^{*} = u^{**}$, hence
\begin{equation*}
\normVptheta{U_{(\tilde{c}^{(n)},\tilde{d}^{(n)})}(t_{n},0)u_{n} -
U_{(\tilde{c},\tilde{d})}(t,0)u_{n}} \to 0
\end{equation*}
as $n \to \infty$, a contradiction.

(2) It follows by (1) and Eq.~\eqref{sobolev-embed-eq1}.
\end{proof}

We are now ready to construct the skew-product semiflow  on $X$ ($X$
is as in \eqref{X-eq}) generated by \eqref{nonauton-eq1} or
\eqref{random-eq1}.  For $t \ge 0$, $(\tilde{c},\tilde{d}) \in Y$,
$u_0 \in X$, put
\begin{equation}
\label{skew-product-semiflow-in-X}
\Pi_t(u_0,(\tilde{c},\tilde{d})) = \Pi(t;u_0,(\tilde{c},\tilde{d}))
:= (U_{(\tilde{c},\tilde{d})}(t,0)u_0,(\tilde{c},\tilde{d}) \cdot
t).
\end{equation}
$\Pi = \{\,\Pi_{t}\,\}_{t \ge 0}$ satisfies the usual algebraic
properties of a semiflow on $X$ : $\Pi_0$ equals the identity on $X$,
and $\Pi_{t} \circ \Pi_{s} = \Pi_{s+t}$ for any $s, t \ge 0$.
Moreover, the continuity of $\Pi$ restricted to $(0,\infty) \times X
\times Y$ follows by Proposition~\ref{joint-continuity-in-X-theta}
and the embedding $X \hookrightarrow L_2(D)$.  (However, we need not
have continuity at $t = 0$.)

\smallskip
Sometimes we write $U_{(\tilde{c},\tilde{d})}(t,s)$ instead of
$U_{(\tilde{c},\tilde{d}) \cdot s}(t-s,0)$, $s \le t$.  The semigroup
property $\Pi_{t} \circ \Pi_{s} = \Pi_{s+t}$ takes in that notation
the following form (see the cocycle property \eqref{cocycle0}):
\begin{equation}
\label{cocycle}
U_{(\tilde{c},\tilde{d})}(t,r) = U_{(\tilde{c},\tilde{d})}(t,s)
U_{(\tilde{c},\tilde{d})}(s,r), \qquad r \le s \le t.
\end{equation}

\begin{proposition}[Continuity in $C(\bar{D})$ at $t = 0$]
\label{continuity-at-time-0}
Let $\theta \in (1/2,1)$ and $p > 1$ be such that $2\theta - p
\not\in \NN$ and $V_p^{\theta} \hookrightarrow C(\bar{D})$. Then for
any $(\tilde{c},\tilde{d}) \in Y$ and $u_0 \in
V_p^{\theta}(\tilde{d})$,
\begin{equation*}
\lVert U_{(\tilde{c},\tilde{d})}(t,0)u_0 - u_0 \rVert_{C(\bar{D})}
\to 0
\end{equation*}
as $t \to 0$.
\end{proposition}
\begin{proof}
It follows from \cite[Theorem 15.1]{Ama} and
Eq.~\eqref{sobolev-embed-eq1}.
\end{proof}

Throughout the rest of this paper, we assume (A1)--(A7).

\section{Strong monotonicity and globally positive solutions}
\label{strong-monotonicity}
In this section, we first show that the skew-product semiflow $\Pi_t$
constructed in the previous section is strongly monotone and then
show that \eqref{nonauton-eq2}  has a unique globally positive
solution, which will be used in next section to define the principal
spectrum and principal Lyapunov exponent of \eqref{nonauton-eq1} and
\eqref{random-eq1}.

Let $X$ be as in \eqref{X-eq}. The Banach space $X$ is ordered by the
standard cone
\begin{equation*}
X^{+} := \{\,u \in X: u(x) \ge 0 \text{ for each } x \in D\,\}.
\end{equation*}
The interior $X^{++}$ of $X^+$ is nonempty, where
\begin{equation*}
X^{++} = \{\,u \in X: u(x) > 0 \text{ for } x \in D \text{ and } (\p
u/\p \nu)(x) < 0 \text{ for } x \in \p D \,\}
\end{equation*}
for the Dirichlet boundary conditions, and
\begin{equation*}
X^{++} = \{\,u \in X: u(x) > 0 \text{ for } x \in \bar{D}\,\}
\end{equation*}
for the Neumann or Robin boundary conditions.

For $u_1, u_2 \in X$, we write $u_1 \le u_2$ if $u_2 - u_1 \in
X^{+}$, $u_1 < u_2$ if $u_1 \le u_2$ and $u_1 \ne u_2$, and $u_1 \ll
u_2$ if $u_2 - u_1 \in X^{++}$.  The symbols $\ge$, $>$ and $\gg$ are
used in the standard way.

We proceed now to investigate the strong monotonicity property of the
solution operator $U_{(\tilde{c},\tilde{d})}(t,0)$. When the
equations \eqref{nonauton-eq1} and \eqref{random-eq1} are in
divergence form, the monotonicity of $U_{(\tilde{c},\tilde{d})}(t,0)$
follows from \cite[Theorem 11.6]{Ama3}. But the strong monotonicity
is not included in \cite[Theorem 11.6]{Ama3}.  Though the
monotonicity for equations in non-divergence form can also be proved
by \cite[Theorem 11.6]{Ama3} after verifying certain conditions,
however for convenience we will give a proof for the monotonicity
directly.  We will prove the strong monotonicity by using the strong
maximum principle and the Hopf boundary point principle for classical
solutions.  But before we do that we have to analyze whether the
existing theory (as presented, e.g., in~\cite{Fri1}) can be applied:
notice that in the Robin case $\tilde{d}$ may change sign.  We show
that coefficient can be made nonnegative by an appropriate change of
variables.

Indeed, consider
\begin{equation}
\label{aux-eq}
\begin{cases}
\displaystyle \frac{\p u^{*}}{\p t} = \sum_{i,j=1}^{N}
a_{ij}(x)\frac{\p^2 u^{*}} {\p x_i\p x_j}, & \quad t > -1,\ x \in D, \\[2ex]
\displaystyle \sum_{i=1}^{N} b_{i}(x) \frac{\p u^{*}}{\p x_{i}} +
u^{*} = 0, & \quad t > -1,\ x \in \p D.
\end{cases}
\end{equation}
Let $p > 1$ and $\theta \in (1/2,1)$ be as in
Proposition~\ref{continuity-at-time-0}.  By the $C^{\infty}$ Urysohn
Lemma (see \cite[Lemma 8.18]{Fol}), there is a nonzero $C^{\infty}$
function $u_0 \colon \RR^{N} \to \RR$ such that $0 \le u_0 \le 1$ on
$D$ and $\supp{u_0} \Subset D$.  Then $u_0 \in V_p^{\theta}(1)$.  Let
$u^*(t,x)$ be the solution of \eqref{aux-eq} with $u^*(-1,x) =
u_0(x)$.  By Proposition~\ref{continuity-at-time-0}
\begin{equation*}
\lVert u^{*}(t,\cdot) - u_0 \rVert_{C(\bar{D})} \to 0 \quad
\text{as} \quad t \to -1^{+}.
\end{equation*}
Hence, the function $u^{*}$ is continuous on $[-1,\infty) \times
\bar{D}$ and satisfies, by Proposition~\ref{classical-solution-prop},
the equation in~\eqref{aux-eq} pointwise on $(-1,\infty) \times D$
and the boundary condition in~\eqref{aux-eq} pointwise on
$(-1,\infty) \times \p D$. Consequently, it follows from the strong
maximum principle and the Hopf boundary point principle for parabolic
equations that $u^{*}(t,x) > 0$ for all $t > -1$ and all $x \in
\bar{D}$.

Now, let $v(t,x) := e^{M u^{*}(t,x)}u(t,x)$, where $M$ is a positive
constant (to be determined later).  Then \eqref{nonauton-eq2} becomes
\begin{equation}
\label{general-eq-smooth1}
\begin{cases}
\disp \frac{\p v}{\p t} = \sum_{i,j=1}^{N} a_{ij}(x)\frac{\p^2 v}{\p
x_i \p x_j} + \sum_{i=1}^{N} \check{a}_{i}(x) \frac{\p v}{\p
x_i} + \check{c}(t,x)v, & \quad t > 0,\ x \in D, \\[2ex]
\disp \sum_{i=1}^{N} b_{i}(x) \frac{\p v}{\p x_{i}} + \check{d}(t,x)
v = 0 & \quad t > 0,\ x \in \p D,
\end{cases}
\end{equation}
where
\begin{align*}
\check{a}_i(x) & := a_{i}(x) - M \biggl(\sum_{j=1}^{N}
\Bigl(a_{ij}(x) \frac{\p u^*}{\p x_j} + a_{ji}(x) \frac{\p u^*}{\p
x_j}\Bigr) \biggr),
\\
\check{c}(t,x) &:= \tilde{c}(t,x) - M \sum_{i=1}^{N} a_i(t,x)
\frac{\p u^*}{\p x_i} + M^{2} \sum_{i,j=1}^{N} a_{ij}(x) \frac{\p
u^*}{\p x_i} \frac{\p u^*}{\p x_j}, \\
\check{d}(t,x) &:= \tilde{d}(t,x) + M u^{*}(t,x).
\end{align*}
We see that for any $(\tilde{c},\tilde{d}) \in Y$ and any $T > 0$,
there is $M = M(T) > 0$ such that $\check{d}(t,x) > 0$ for $t \in
[0,T]$ and $x \in \bar{D}$.  Observe that, since the mapping
$[\,[0,\infty) \ni t \mapsto u^{*}(t,\cdot) \in C^{1}(\bar{D})\,]$ is
continuous by Proposition~\ref{joint-continuity-in-X-theta}, the
coefficients $\check{a}_i$ and $\check{c}$ are bounded on $[0,T]
\times \bar{D}$ and the coefficient $\check{d}$ is bounded on $[0,T]
\times \p D$.

Consequently, we have the following result.
\begin{theorem}[Strong monotonicity]
\label{strong-positivity-class}
Let $u_1, u_2 \in L_2(D)$.  If $u_1 \ne u_2$ and $u_1(x) \le u_2(x)$
for a.e.~$x \in D$, then
\begin{itemize}
\item[{\rm (i)}]
\begin{equation*}
(U_{(\tilde{c},\tilde{d})}(t,0)u_1)(x) <
(U_{(\tilde{c},\tilde{d})}(t,0)u_2)(x) \quad \text{for }
(\tilde{c},\tilde{d}) \in Y,\ t
> 0 \text{ and }x \in D
\end{equation*}
and
\begin{equation*}
\frac{\p}{\p \nu}(U_{(\tilde{c},\tilde{d})}(t,0)u_1)(x) >
\frac{\p}{\p \nu} (U_{(\tilde{c},\tilde{d})}(t,0)u_2)(x) \quad
\text{for } (\tilde{c},\tilde{d}) \in Y,\ t > 0 \text{ and }x \in \p
D
\end{equation*}
in the Dirichlet case,
\item[{\rm (ii)}]
\begin{equation*}
(U_{(\tilde{c},\tilde{d})}(t,0)u_1)(x) <
(U_{(\tilde{c},\tilde{d})}(t,0)u_2)(x) \quad \text{for }
(\tilde{c},\tilde{d}) \in Y,\ t
> 0 \text{ and }x \in \bar{D}
\end{equation*}
in the Neumann or Robin case.
\end{itemize}
\end{theorem}
\begin{proof}
Fix $(\tilde{c},\tilde{d}) \in Y$.  Assume that $u_1, u_2 \in X$ and
$u_1 < u_2$. For any given $T > 0$, in the case of the Robin boundary
conditions let $M > 0$ be such that $\check{d}(t,x) > 0$ for all $t
\in [0,T]$ and all $x \in \bar{D}$, where $\check{d}$ is as in the
reasoning above the statement of the present proposition (in the case
of the Dirichlet or Neumann boundary conditions put $M = 0$).  Define
$v_0(x) := e^{M u^{*}(0,x)}(u_2(x) - u_1(x))$, $x \in \bar{D}$.

Let $\theta \in (1/2,1)$ and $p > 1$ be as in
Proposition~\ref{continuity-at-time-0}.  We claim that there is a
sequence $(v^{(n)})_{n=1}^{\infty} \subset V_p^{\theta}(\check{d})$
such that $v^{(n)}(x) \ge 0$ ($n = 1,2,\dots$, $x \in D$), $v^{(n)}
\not\equiv 0$ ($n = 1,2,\dots$),  and $\lim_{n\to\infty}\norm{v^{(n)}
- v_0} = 0$.

First note that there is a sequence $(v_0^{(n)})_{n=1}^\infty$ of
simple functions such that
\begin{equation*}
0 \le v_0^{(1)}(x) \le v_0^{(2)}(x) \le \dots \le v_0(x) \quad
\text{for~a.e. } x \in D,
\end{equation*}
and
\begin{equation*}
v_0^{(n)}(x) \to v_0(x) \quad \text{as} \quad n \to \infty, \quad
\text{for~a.e. } x \in D,
\end{equation*}
and $v_0^{(n)} \to v_0$ uniformly on any set on which $v_0$ is
bounded.  It is therefore sufficient to prove the claim for the case
that $v_0 = \chi_E$, where $E \subset D$ is a Lebesgue measurable
set.

Now assume $v_0 = \chi_E$, where $E \subset D$ is a Lebesgue
measurable set.  For $\epsilon_n := \frac{1}{4n^2}$, choose a compact
set $K \subset E$ and an open set $U \supset K$ such that $U \Subset
D$, $\abs{E \setminus K} < \epsilon_n$ and $\abs{U \setminus K} <
\epsilon_n$, where here $\abs{\cdot}$ denotes the Lebesgue measure of
a set.  Then, by the $C^{\infty}$ Urysohn Lemma (see \cite[Lemma
8.18]{Fol}), there is a $C^{\infty}$ function $v^{(n)} \colon \RR^{N}
\to \RR$ such that $0 \le v^{(n)} \le 1$ on $D$, $v^{(n)} \equiv 1$
on $K$ and $\supp{v^{(n)}} \subset U$.  It then follows that
\begin{equation*}
\norm{v^{(n)} - v_0} \le \abs{U \setminus K}^{1/2} + \abs{E \setminus
K}^{1/2} < \frac{1}{n} \to 0
\end{equation*}
as $n \to \infty$.  Moreover, since $\supp{v^{(n)}} \subset U \Subset
D$, we also have $v^{(n)} \in V_p^\theta(\check{d})$. The claim is
thus proved.
Denote by $v(t,\cdot;v_0)$ and $v(t,\cdot;v^{(n)})$ the solutions
of~\eqref{general-eq-smooth1} with $v(0,\cdot;v_0) = v_0(\cdot)$ and
$v(0,\cdot;v^{(n)}) = v^{(n)}(\cdot)$ ($n = 1,2,\dots$),
respectively.

By Proposition~\ref{continuity-at-time-0},
\begin{equation*}
\lVert v(t,\cdot;v^{(n)}) - v^{(n)} \rVert_{C(\bar{D})} \to 0
\end{equation*}
as $t \to 0^{+}$.  We can thus apply the strong comparison principle
for parabolic equations to conclude that
\begin{equation*}
v(t,x;v^{(n)}) > 0 \quad \text{for} \quad t \in (0,T],\ x \in D, \ n
= 1,2,\dots.
\end{equation*}
This together with Proposition~\ref{norm-continuity-prop-in-X-theta}
implies that
\begin{equation*}
v(t,x;v_0) \ge 0 \quad \text{for} \quad t \in (0,T],\ x \in D.
\end{equation*}

By Proposition~\ref{classical-solution-prop}, for any $n = 2,3,
\dots$ the function $v(\cdot,\cdot;v_0)$ is continuous on $[T/n,T]$,
satisfies the equation in~\eqref{general-eq-smooth1} pointwise on
$(T/n,T] \times D$ and satisfies the boundary condition
in~\eqref{general-eq-smooth1} pointwise on $(T/n,T] \times \p D$.
Further, from Proposition~\ref{existence-prop} and the nonnegativity
of $v$ it follows that for $n$ sufficiently large there is $x_n \in
D$ such that $v(T/n,x_n;v_0) > 0$.  An application of the strong
maximum principle for parabolic equations gives $v(t,x;v_0) > 0$ for
each $t \in (0,T]$ and each $x \in D$.

In the Dirichlet boundary condition case, suppose to the contrary
that there are $t^{*} \in (0,T]$ and $x^{*} \in \p D$ such that
$\frac{\p}{\p \nu}(v(t^{*},x^{*};v_0)) = 0$.  But this contradicts
the Hopf boundary point principle applied to $v$ restricted to
$[t^{*}/2,t^{*}] \times \bar{D}$.  Hence $\frac{\p}{\p
\nu}(v(t,x;v_0)) < 0$ for any $t \in (0,T]$ and any $x \in \p D$.
This completes the proof in that case, since $v(t,x;v_0) =
(U_{\tilde{a}}(t,0)u_2)(x) - (U_{\tilde{a}}(t,0)u_1)(x)$ for any $t
\in (0,T]$ and $x \in \bar{D}$.

Suppose to the contrary that, in the Neumann or Robin boundary
condition case, there are $t^{*} \in (0,T]$ and $x^{*} \in \p D$ such
that $v(t^{*},x^{*};v_0) = 0$.  It follows from the Hopf boundary
point principle (applied to $v$ restricted to $[t^{*}/2,t^{*}] \times
\bar{D}$) that $\sum_{i=1}^{N} b_{i}(x^{*}) \frac{\p}{\p
x_{i}}v(t^{*},x^{*};v_0) < 0$, which is incompatible with the
boundary condition.  Hence $v(t,x;v_0) > 0$ for all $t \in (0,T]$ and
all $x \in \bar{D}$.  Since $v(t,x;v_0) = e^{M u^{*}(t,x)} (
(U_{\tilde{a}}(t,0)u_2)(x) - (U_{\tilde{a}}(t,0)u_1)(x))$ for any $t
\in (0,T]$ and $x \in \bar{D}$, this completes the proof.

(It is to be remarked that in Eqs.~\eqref{aux-eq}
and~\eqref{general-eq-smooth1} their coefficients may not belong to
$Y$, so {\em formally\/} we cannot apply propositions from
Section~\ref{skew-product} in those cases.  This should not cause any
misunderstanding.)
\end{proof}

By Theorem~\ref{strong-positivity-class} we have the following {\em
strong monotonicity\/}:

\vskip1.5ex {\em For $(\tilde{c},\tilde{d}) \in Y$, $u_1, u_2 \in X$
and $t > 0$, if $u_1 < u_2$ then $U_{(\tilde{c},\tilde{d})}(t,0)u_1
\ll U_{(\tilde{c},\tilde{d})}(t,0)u_2$.}

\vskip3ex The theory of existence and uniqueness of globally positive
solutions can then be extended to our case.  Below, we collect its
basic concepts and facts.

\begin{definition}
\label{globally-positive-def}
For $(\tilde{c},\tilde{d}) \in Y$, we say that a global solution $v =
v(t,x)$ of~\eqref{nonauton-eq2} is a {\em globally positive
solution\/} of~\eqref{nonauton-eq2} if $v(t,x) > 0$ for all $t \in
\RR$ and all $x \in D$.
\end{definition}

We shall consider now the problem of existence of globally positive
solutions.
\begin{theorem}
\label{globally-positive-existence}
There exist
\begin{itemize}
\item
a continuous function $w \colon Y \to X^{++}$,
$\norm{w((\tilde{c},\tilde{d}))} = 1$ for each
$(\tilde{c},\tilde{d}) \in Y$, and
\item
a continuous function $w^{*} \colon Y \to L_2(D)$,
$\norm{w^{*}((\tilde{c},\tilde{d}))} = 1$ for each
$(\tilde{c},\tilde{d}) \in Y$, and such that for each
$(\tilde{c},\tilde{d}) \in Y$, $w^{*}((\tilde{c},\tilde{d}))(x) >
0$ for a.e.~$x \in D$,
\end{itemize}
having the following properties:
\begin{itemize}
\item[{\rm (i)}]
For each $(\tilde{c},\tilde{d}) \in Y$ the function
$v_{(\tilde{c},\tilde{d})} = v(t,x;\tilde{c},\tilde{d})$ given by
\begin{equation}
\label{globally-positive-formula}
v(t,\cdot;\tilde{c},\tilde{d}) :=
\begin{cases}
U_{(\tilde{c},\tilde{d})}(t,0) w((\tilde{c},\tilde{d}))
& \quad \text{for } t \ge 0, \\[1ex]
\disp \frac{w((\tilde{c},\tilde{d}) \cdot
t)}{\norm{U_{(\tilde{c},\tilde{d})\cdot
t}(-t,0)w((\tilde{c},\tilde{d}) \cdot t)}} & \quad \text{for } t <
0,
\end{cases}
\end{equation}
is a globally positive solution of~\eqref{nonauton-eq2}.
\item[{\rm (ii)}]
Let, for some $(\tilde{c},\tilde{d}) \in Y$, $v = v(t,x)$ be a
globally positive solution of~\eqref{nonauton-eq2}.  Then there
exists a constant $\beta > 0$ such that $v(t,x) = {\beta}
v(t,x;\tilde{c},\tilde{d})$ for each $t \in \RR$ and each $x \in
D$.
\item[{\rm (iii)}]
There are constants $C > 0$ and $\mu > 0$ such that
\begin{equation}
\label{exp-sep-formula-1} \norm{U_{(\tilde{c},\tilde{d})}(t,0) u_0}
\le C e^{{-\mu}t} \norm{U_{(\tilde{c},\tilde{d})}(t,0)
w((\tilde{c},\tilde{d}))}
\end{equation}
for any $(\tilde{c},\tilde{d}) \in Y$, $t > 0$ and $u_0 \in
L_2(D)$ with $\norm{u_0} = 1$ and $\langle u_0,
w^{*}((\tilde{c},\tilde{d})) \rangle = 0$.
\item[{\rm (iv)}]
There are constants $C' > 0$ and $\mu > 0$ such that
\begin{equation}
\label{exp-sep-formula-2} \lVert U_{(\tilde{c},\tilde{d})}(t,0) u_0
\rVert_{X} \le C' e^{{-\mu}t} \lVert U_{(\tilde{c},\tilde{d})}(t,0)
w((\tilde{c},\tilde{d})) \rVert_{X}
\end{equation}
for any $(\tilde{c},\tilde{d}) \in Y$, $t \ge 1$ and $u_0 \in
L_2(D)$ with $\norm{u_0} = 1$ and $\langle u_0,
w^{*}((\tilde{c},\tilde{d})) \rangle = 0$.
\end{itemize}
\end{theorem}
\begin{proof}
We start by considering a discrete-time dynamical system on the
product bundle $X \times Y$ ($X$ is a fiber, $Y$ is the base space):
\begin{equation}
\Pi_n(u_0,(\tilde{c},\tilde{d})) :=
(U_{(\tilde{c},\tilde{d})}(n,0)u_0,(\tilde{c},\tilde{d}) \cdot n),
\qquad u_0 \in X,\ (\tilde{c},\tilde{d}) \in Y,\ n = 1,2,3,\dots.
\end{equation}
Proposition~\ref{norm-continuity-prop-in-X-theta} and
Theorem~\ref{strong-positivity-class} allow us to use the results
contained in~\cite{PoTer} to conclude that there are continuous
functions $\tilde{w} \colon Y \to X$, $\tilde{w}^{*} \colon Y \to
X^{*}$, $\lVert \tilde{w}((\tilde{c},\tilde{d}))\rVert_{X} = \lVert
\tilde{w}^{*}((\tilde{c},\tilde{d}))\rVert_{X^{*}} = 1$ for each
$\tilde{a} \in Y$, such that (we write $X_1((\tilde{c},\tilde{d})) :=
\spanned{\tilde{w}((\tilde{c},\tilde{d}))}$,
$X_2((\tilde{c},\tilde{d})) :=
\nullspace{(\tilde{w}^{*}((\tilde{c},\tilde{d})))}$, where
$\nullspace$ stands for the nullspace of an element of $X^{*}$)
\begin{itemize}
\item[(a)]
$\tilde{w}((\tilde{c},\tilde{d})) \in X^{++}$, for each
$(\tilde{c},\tilde{d}) \in Y$.
\item[(b)]
$(v, \tilde{w}^{*}((\tilde{c},\tilde{d})))_{X,X^{*}} > 0$ for
each $(\tilde{c},\tilde{d}) \in X$ and each nonzero $v \in
X^{+}$.  It follows that $X_2((\tilde{c},\tilde{d})) \cap X^{+} =
\{0\}$, for each $(\tilde{c},\tilde{d}) \in Y$.
\item[(c)]
For each $(\tilde{c},\tilde{d}) \in Y$ there is $d_1 =
d_1((\tilde{c},\tilde{d})) > 0$ such that
$U_{(\tilde{c},\tilde{d})}(1,0)\tilde w((\tilde{c},\tilde{d})) =
d_{1}\tilde w((\tilde{c},\tilde{d}) \cdot 1)$.  It follows that
$U_{(\tilde{c},\tilde{d})}(1,0) X_1((\tilde{c},\tilde{d})) =
X_1((\tilde{c},\tilde{d}) \cdot 1)$.
\item[(d)]
For each $(\tilde{c},\tilde{d}) \in Y$ there is $d_1^{*} =
d_1^{*}((\tilde{c},\tilde{d})) > 0$ such that\break
$(U_{(\tilde{c},\tilde{d})}(1,0))^{*} \tilde
w^{*}((\tilde{c},\tilde{d}) \cdot 1) = d_1^{*} \tilde
w^*((\tilde{c},\tilde{d}))$, where
$(U_{(\tilde{c},\tilde{d})}(1,0))^{*} \colon X^{*} \to X^{*}$
stands for the linear operator dual to
$U_{(\tilde{c},\tilde{d})}(1,0)$. It follows that
\break$U_{(\tilde{c},\tilde{d})}(1,0) X_2((\tilde{c},\tilde{d}))
\subset X_2((\tilde{c},\tilde{d}) \cdot 1)$, for any
$(\tilde{c},\tilde{d}) \in Y$.
\item[(e)]
There are constants $\tilde{C} > 0$ and $0 < \gamma < 1$ such
that
\begin{equation}
\label{exp-sep-1}
\lVert U_{(\tilde{c},\tilde{d})}(n,0) u_0 \rVert_{X} \le \tilde{C}
{\gamma}^{n} \lVert U_{(\tilde{c},\tilde{d})}(n,0)
\tilde{w}((\tilde{c},\tilde{d})) \rVert_{X}
\end{equation}
for any $(\tilde{c},\tilde{d}) \in Y$, any $u_0 \in
X_2((\tilde{c},\tilde{d}))$ with $\lVert u_0 \rVert_{X} = 1$ and
any $n \in \NN$.
\end{itemize}
Put $w((\tilde{c},\tilde{d})) := \tilde{w}((\tilde{c},\tilde{d}))/
\norm{\tilde{w}((\tilde{c},\tilde{d}))}$, $(\tilde{c},\tilde{d}) \in
Y$.  As $X$ embeds continuously in $L_2(D)$, the function $w \colon Y
\to X$ is continuous.  Further, put $w^{*}((\tilde{c},\tilde{d})) :=
\tilde{w}^{*}((\tilde{c},\tilde{d}))/
\norm{\tilde{w}^{*}((\tilde{c},\tilde{d}))}$, $(\tilde{c},\tilde{d})
\in Y$.  From Proposition~\ref{norm-continuity-prop-in-X-theta} it
follows that the mapping $[\, Y \ni (\tilde{c},\tilde{d}) \mapsto
(U_{(\tilde{c},\tilde{d})}(1,0))^{*} \in \mathcal{L}(X^{*},L_2(D))
\,]$ is continuous, too, so we obtain with the help of (d) that
$w^{*} \colon Y \to L_2(D)$ is well defined and continuous.

By the definition of the dual operator,
\begin{align*}
d_1^{*}((\tilde{c},\tilde{d})) \cdot (v,
\tilde{w}^{*}((\tilde{c},\tilde{d})))_{X,X^{*}} & = (v,
(U_{(\tilde{c},\tilde{d})}(1,0))^{*}
\tilde{w}^{*}((\tilde{c},\tilde{d}) \cdot 1))_{X,X^{*}} \\
& = (U_{(\tilde{c},\tilde{d})}(1,0)v,
\tilde{w}^{*}((\tilde{c},\tilde{d}) \cdot 1))_{X,X^{*}}
\end{align*}
for each $(\tilde{c},\tilde{d}) \in Y$ and each $v \in X$.  As
$\tilde{w}^{*}((\tilde{c},\tilde{d}))$ is a bounded linear functional
on $L_2(D)$ and $X$ is dense in $L_2(D)$, we conclude that
\begin{align*}
d_1^{*}((\tilde{c},\tilde{d})) \cdot \langle v,
\tilde{w}^{*}((\tilde{c},\tilde{d})) \rangle & = \langle v,
(U_{(\tilde{c},\tilde{d})}(1,0))^{*}\tilde{w}^{*}((\tilde{c},\tilde{d})
\cdot 1) \rangle \\
& = \langle U_{(\tilde{c},\tilde{d})}(1,0)v,
\tilde{w}^{*}((\tilde{c},\tilde{d}) \cdot 1) \rangle
\end{align*}
for each $(\tilde{c},\tilde{d}) \in Y$ and each $v \in L_2(D)$.

We prove now that $w^{*}((\tilde{c},\tilde{d}))(x) > 0$ for a.e.~$x
\in D$, or, which is equivalent, that
$\tilde{w}^{*}((\tilde{c},\tilde{d}))(x) > 0$ for a.e.~$x \in D$.
Suppose first that for some $(\tilde{c},\tilde{d}) \in X$ there are
$D_{+}, D_{-} \subset D$ of positive Lebesgue measure such that
$\tilde{w}^{*}((\tilde{c},\tilde{d}))(x) > 0$ for $x \in D_{+}$,
$\tilde{w}^{*}((\tilde{c},\tilde{d}))(x) < 0$ for $x \in D_{-}$, and
$\tilde{w}^{*}((\tilde{c},\tilde{d}))(x) = 0$ for $x \in D \setminus
(D_{+} \cup D_{-})$.  Define $v \in L_2(D)$ to be the simple function
equal to $1/\int_{D_{+}}\tilde{w}^{*}((\tilde{c},\tilde{d}))(x)\,dx$
on $D_{+}$, equal to
$-1/\int_{D_{-}}\tilde{w}^{*}((\tilde{c},\tilde{d}))(x)\,dx$ on
$D_{-}$, and equal to zero elsewhere.  We have
\begin{multline*}
0 = d_1^{*}((\tilde{c},\tilde{d})) \cdot \langle v,
\tilde{w}^{*}((\tilde{c},\tilde{d})) \rangle = \langle v,
(U_{(\tilde{c},\tilde{d})}(1,0))^{*}
\tilde{w}^{*}((\tilde{c},\tilde{d}) \cdot 1)
\rangle \\
= \langle U_{(\tilde{c},\tilde{d})}(1,0)v,
\tilde{w}^{*}((\tilde{c},\tilde{d}) \cdot 1) \rangle =
(U_{(\tilde{c},\tilde{d})}(1,0)v, \tilde{w}^{*}((\tilde{c},\tilde{d})
\cdot 1))_{X,X^{*}}.
\end{multline*}
By Theorem~\ref{strong-positivity-class},
$U_{(\tilde{c},\tilde{d})}(1,0)v \in X^{++}$.  This contradicts (b).
Suppose now that for some $(\tilde{c},\tilde{d}) \in X$ there are
$D_{+}, D_{0} \subset D$ of positive Lebesgue measure such that
$\tilde{w}^{*}((\tilde{c},\tilde{d}))(x) > 0$ for $x \in D_{+}$ and
$\tilde{w}^{*}((\tilde{c},\tilde{d}))(x) = 0$ for $x \in D_{0}$, and
the complement of the union $D_{+} \cup D_{0}$ in $D$ has Lebesgue
measure zero.  We repeat the above construction, this time with $v$
equal to zero on $D_{+}$ and equal to one on $D_{0}$.

Fix $(\tilde{c},\tilde{d}) \in Y$.  The fact that if there exists a
globally positive solution of~\eqref{nonauton-eq2} then it is unique
up~to multiplication by a positive constant is proved for the
Dirichlet case in \cite{HuPoSa2}, and for the Neumann and Robin case
in~\cite{Hu1}.  We proceed now to the construction of a globally
positive solution.

We define first the trace of a positive solution
$v(t,x;\tilde{c},\tilde{d})$ on $\ZZ$:
\begin{equation*}
v(k,\cdot;\tilde{c},\tilde{d}) :=
\begin{cases}
U_{(\tilde{c},\tilde{d})}(k,0) w((\tilde{c},\tilde{d})) & \quad
\text{for } k = 0,1,2,3, \dots,
\\[1ex]
\disp \frac{w((\tilde{c},\tilde{d}) \cdot
k)}{\norm{U_{(\tilde{c},\tilde{d}) \cdot
k}(-k,0)w((\tilde{c},\tilde{d}) \cdot k)}} & \quad \text{for } k =
\dots, -3, -2, -1.
\end{cases}
\end{equation*}
It follows from (a) and (c) that
\begin{equation}
\label{glob-exist-1} U_{(\tilde{c},\tilde{d})}(l+k,k) v(k,\cdot;
\tilde{c},\tilde{d}) = v(k+l,\cdot; \tilde{c},\tilde{d})
\end{equation}
for any $k \in \ZZ$ and any nonnegative integer $l$.  Also,
$\norm{v(0,\cdot; \tilde{c},\tilde{d})} = 1$. We extend $v$ to a
function defined on $(-\infty,\infty)$ by putting
\begin{equation}
\label{glob-exist-2} v(t,\cdot;\tilde{c},\tilde{d}) :=
U_{(\tilde{c},\tilde{d})}(t,\floor{t})
v(\floor{t},\cdot;\tilde{c},\tilde{d}), \qquad t \in \RR \setminus
\ZZ,
\end{equation}
where $\floor{t}$ denotes the greatest integer less than or equal to
$t$.  To check that the function so defined is indeed a global
solution we need to show that
\begin{equation}
\label{glob-exist-3} v(s+t,\cdot;\tilde{c},\tilde{d}) =
U_{(\tilde{c},\tilde{d})}(s+t,t) v(t,\cdot;\tilde{c},\tilde{d})
\qquad \text{for any } t \in \RR \text{ and any }s \ge 0
\end{equation}
(see Definition~\ref{global-solution-def} and Eq.~\eqref{cocycle}).
We write
\begin{align*}
v(s+t,\cdot;\tilde{c},\tilde{d}) & =
U_{(\tilde{c},\tilde{d})}(s+t,\floor{s+t})
v(\floor{s+t},\cdot;\tilde{c},\tilde{d}) &\text{by \eqref{glob-exist-2}}\\
& = U_{(\tilde{c},\tilde{d})}(s+t,\floor{s+t})
U_{(\tilde{c},\tilde{d})}(\floor{s+t},\floor{t})
v(\floor{t},\cdot;\tilde{c},\tilde{d}) &\text{by \eqref{glob-exist-1}}\\
& = U_{(\tilde{c},\tilde{d})}(s+t,t)
U_{(\tilde{c},\tilde{d})}(t,\floor{t})
v(\floor{t},\cdot;\tilde{c},\tilde{d}) &\text{by \eqref{cocycle}} \\
& = U_{(\tilde{c},\tilde{d})}(s+t,t) v(t,\cdot;\tilde{c},\tilde{d})
&\text{by \eqref{glob-exist-2}}.
\end{align*}
The fact that $v(t,\cdot;\tilde{c},\tilde{d}) \in X^{++}$ for each $t
\in \RR$ is a consequence of the construction of $v$ and of
Theorem~\ref{strong-positivity-class}.

Formula~\eqref{globally-positive-formula} for $t \ge 0$ is
straightforward.  It follows from the uniqueness of globally positive
solutions that
\begin{equation*}
v(t,\cdot;\tilde{c},\tilde{d}) =
\norm{v(t,\cdot;\tilde{c},\tilde{d})} \, w((\tilde{c},\tilde{d})
\cdot t), \qquad t \in (-\infty,\infty).
\end{equation*}
From~\eqref{glob-exist-3} we obtain, for any $t < 0$, that
\begin{align*}
1 = \norm{v(0,\cdot;\tilde{c},\tilde{d})} & =
\norm{U_{(\tilde{c},\tilde{d})\cdot t}(-t,0)
v(t,\cdot;\tilde{c},\tilde{d})} \\
& = \norm{v(t,\cdot;\tilde{c},\tilde{d})} \,
\norm{U_{(\tilde{c},\tilde{d})\cdot t}(-t,0) w((\tilde{c},\tilde{d})
\cdot t)},
\end{align*}
which concludes the proof of
formula~\eqref{globally-positive-formula}.

We proceed now to the proof of part (iii).  Denote by $M_1$ the norm
of the embedding $X \hookrightarrow L_2(D)$.  Moreover, by the
compactness of $Y$ and the continuity of $\tilde{w}$ there is $M_2 >
0$ such that $\lVert \tilde{w}((\tilde{c},\tilde{d})) \rVert_{X} \le
M_2 \norm{w((\tilde{c},\tilde{d}))}$ for all $(\tilde{c},\tilde{d})
\in Y$.

Take $u_0 \in L_2(D)$ such that $\norm{u_0} = 1$ and $\langle u_0,
w^{*}((\tilde{c},\tilde{d})) \rangle = 0$.  It follows from (d) that
$\langle U_{(\tilde{c},\tilde{d})}(1,0)u_0,
w^{*}((\tilde{c},\tilde{d}) \cdot 1) \rangle = 0$.  As
$U_{(\tilde{c},\tilde{d})}(1,0)u_0 \in X$, one has
$U_{(\tilde{c},\tilde{d})}(1,0)u_0 \in X_2(\tilde{a} \cdot 1)$. This
allows us to estimate, for $n = 2,3,4, \dots$,
\begin{align*}
\norm{U_{(\tilde{c},\tilde{d})}(n,0) u_0} & \le M_1 \lVert
U_{(\tilde{c},\tilde{d})}(n,1) (U_{(\tilde{c},\tilde{d})}(1,0)u_0)
\rVert_{X} &  \text{ by \eqref{cocycle}}
\\
& \le M_1 \tilde{C} {\gamma}^{n-1} \lVert
U_{(\tilde{c},\tilde{d})}(n,1) \tilde{w}((\tilde{c},\tilde{d}) \cdot
1) \rVert_{X} \lVert U_{(\tilde{c},\tilde{d})}(1,0)u_0
\rVert_{X} &  \text{ by \eqref{exp-sep-1}} \\
& = \frac{M_1 \tilde{C}}{\gamma} {\gamma}^n \frac{\lVert
U_{(\tilde{c},\tilde{d})}(n,0) \tilde{w}((\tilde{c},\tilde{d}))
\rVert_{X}}{\lVert U_{(\tilde{c},\tilde{d})}(1,0)
\tilde{w}((\tilde{c},\tilde{d})) \rVert_{X}} \lVert
U_{(\tilde{c},\tilde{d})}(1,0)u_0 \rVert_{X} &  \text{ by
\eqref{cocycle}} \\
& \le \frac{M_{1} M_{2} D_{1} \tilde{C}}{D_2 \gamma} {\gamma}^{n}
\norm{U_{(\tilde{c},\tilde{d})}(n,0) w((\tilde{c},\tilde{d}))},
\end{align*}
where $D_1 := \sup\{\,\lVert U_{(\tilde{c},\tilde{d})}(1,0)
\rVert_{L_2(D),X}: (\tilde{c},\tilde{d}) \in Y \,\} < \infty$,
\newline
$D_2 := \inf\{\,\lVert
U_{(\tilde{c},\tilde{d})}(1,0)\tilde{w}((\tilde{c},\tilde{d}))
\rVert_{X}: (\tilde{c},\tilde{d}) \in Y \,\} > 0$.

Clearly,  $\norm{U_{(\tilde{c},\tilde{d})}(1,0) u_0} \le
\frac{M_1M_2D_1}{D_2} \norm{U_{(\tilde{c},\tilde{d})}(1,0)
w((\tilde{c},\tilde{d}))}$ for all $(\tilde{c},\tilde{d}) \in Y$ and
all $u_0 \in L_2(D)$ with $\norm{u_0} = 1$ and $\langle u_0,
w^{*}((\tilde{c},\tilde{d})) \rangle = 0$.

As a consequence we obtain the existence of $\bar{C} = \frac{M_{1}
M_{2} D_{1} }{D_{2} \gamma} \max\{\tilde{C}, 1\}$ such that
\begin{equation}
\label{exp-sep-discrete} \norm{U_{(\tilde{c},\tilde{d})}(n,0) u_0}
\le \bar{C} \gamma^n \norm{U_{(\tilde{c},\tilde{d})}(n,0)
w((\tilde{c},\tilde{d}))}
\end{equation}
for any $(\tilde{c},\tilde{d}) \in Y$, any $n \in \NN$ and any $u_0
\in L_2(D)$ satisfying $\norm{u_0} = 1$ and $\langle u_0,
w^*((\tilde{c},\tilde{d})) \rangle = 0$.

To show \eqref{exp-sep-formula-1} we notice that
\begin{align*}
\norm{U_{(\tilde{c},\tilde{d})}(t,0) u_0} & =
\norm{U_{(\tilde{c},\tilde{d})}(t,\floor{t})
(U_{(\tilde{c},\tilde{d})}(\floor{t},0) u_0)} \\
& \le D_3 \norm{U_{(\tilde{c},\tilde{d})}(\floor{t},0) u_0} \\
& \le D_3 \bar{C} \gamma^{\floor{t}}
\norm{U_{(\tilde{c},\tilde{d})}(\floor{t},0)
w((\tilde{c},\tilde{d}))} \qquad \qquad \qquad \qquad \qquad
\text{by \eqref{exp-sep-discrete}} \\
& \le D_{3} \bar{C} {\gamma}^{\floor{t}}
\frac{1}{\norm{U_{(\tilde{c},\tilde{d}) \cdot
\floor{t}}(t-\floor{t},0) w((\tilde{c},\tilde{d}) \cdot \floor{t})}}
\, \norm{U_{\tilde{a}}(t,0) w((\tilde{c},\tilde{d}))} \\
& \le \frac{D_{3} \bar{C}}{{\gamma} D_4} \gamma^{t}
\norm{U_{(\tilde{c},\tilde{d})}(t,0) w((\tilde{c},\tilde{d}))}
\end{align*}
for any $(\tilde{c},\tilde{d}) \in Y$, $t \ge 1$ and any $u_0 \in
L_2(D)$ with $\norm{u_0} = 1$ and $\langle u_0,
w^*((\tilde{c},\tilde{d})) \rangle = 0$, where $D_3 :=
\sup\{\,\norm{U_{(\tilde{c},\tilde{d})}(t,0)}: t \in [0,1],\
(\tilde{c},\tilde{d}) \in Y\,\} < \infty$ and $D_4 := \inf\{\,
\norm{U_{(\tilde{c},\tilde{d})}(t,0)\tilde{w}((\tilde{c},\tilde{d}))}:
t \in [0,1],\ (\tilde{c},\tilde{d}) \in Y \,\} > 0$.

Clearly, $\norm{U_{(\tilde{c},\tilde{d})}(t,0) u_0} \le
\frac{D_3}{D_4} \norm{U_{(\tilde{c},\tilde{d})}(t,0)
w((\tilde{c},\tilde{d}))}$ for all $(\tilde{c},\tilde{d}) \in Y$, all
$t \in [0,1]$ and all $u_0 \in L_2(D)$ with $\norm{u_0} = 1$ and
$\langle u_0, w^*((\tilde{c},\tilde{d})) \rangle = 0$.

This proves \eqref{exp-sep-formula-1}, with $C = \frac{D_3}{D_4
\gamma} \max\{\bar{C}, 1\}$ and $\mu = -\ln{\lambda}$.

To prove~\eqref{exp-sep-formula-2} we estimate, for $u_0 \in L_2(D)$
with $\norm{u_0} = 1$ and $\langle u_0, w^{*}((\tilde{c},\tilde{d}))
\rangle = 0$, and $t \ge 1$,
\begin{align*}
\lVert U_{(\tilde{c},\tilde{d})}(t,0) u_0 \rVert_{X} & = \lVert
U_{(\tilde{c},\tilde{d})}(t,t-1)
(U_{(\tilde{c},\tilde{d})}(t-1,0)u_0) \rVert_{X} \\
& \le D_{1} \norm{U_{(\tilde{c},\tilde{d})}(t-1,0)u_0} \\
& \le D_{1} C e^{-\mu(t-1)}
\norm{U_{\tilde{a}}(t-1,0)w((\tilde{c},\tilde{d}))}
& \qquad \text{ by \eqref{exp-sep-formula-1}} \\
& \le \frac{D_1 C e^{\mu}}{D_5} e^{-{\mu}t}
\norm{U_{(\tilde{c},\tilde{d})}(t,0)w((\tilde{c},\tilde{d}))} \\
& \le \frac{D_1 M_1 C e^{\mu}}{D_5} e^{-{\mu}t} \lVert
U_{(\tilde{c},\tilde{d})}(t,0)w((\tilde{c},\tilde{d}))\rVert_{X},
\end{align*}
where $M_1 := \sup\{\,\norm{u}: u \in X, \lVert u \rVert_{X} \le
1\,\}$, $D_1 := \sup\{\,\lVert U_{(\tilde{c},\tilde{d})}(1,0)
\rVert_{L_2(D),X}: $ $(\tilde{c},\tilde{d}) \in Y \,\}$ and $D_5 :=
\inf\{\,\norm{U_{(\tilde{c},\tilde{d})}(1,0)w((\tilde{c},\tilde{d}))}:
(\tilde{c},\tilde{d}) \in Y \,\}$.
\end{proof}
For other approaches to the question of existence and/or uniqueness
of globally positive solutions the reader can consult
also~\cite{Hu2}, \cite{HuPo}, \cite{HuPoSa1}, \cite{Mi1}, \cite{Mi2},
\cite{Po}.

\begin{theorem}
\label{pre-bounds-of-w-thm}
\begin{itemize}
\item[{\rm (1)}]
In the Dirichlet boundary condition case, there is $M > 0$ such
that
\begin{equation*}
w((\tilde{c},\tilde{d}))(x) \le M \quad \text{for any } x \in D
\text{ and any } (\tilde{c},\tilde{d}) \in Y.
\end{equation*}
Further, for each compact $D_0 \Subset D$ there is $m = m(D_0) >
0$ such that
\begin{equation*}
w((\tilde{c},\tilde{d}))(x) \ge m(D_0) \quad \text{for any } x \in
D_0 \text{ and any } (\tilde{c},\tilde{d}) \in Y.
\end{equation*}
\item[{\rm (2)}]
In the Neumann or Robin boundary condition case, there are $M, m
> 0$ such that
\begin{equation*}
m \le w((\tilde{c},\tilde{d}))(x) \le M \quad \text{for any } x \in
D \text{ and any } (\tilde{c},\tilde{d}) \in Y.
\end{equation*}
\end{itemize}
\end{theorem}
\begin{proof}
It follows in a standard way from the compactness of $Y$ and from the
fact that $w((\tilde{c},\tilde{d})) \in X^{++}$ for each
$(\tilde{c},\tilde{d}) \in Y$.
\end{proof}

\begin{theorem}
\label{Holder-continuity-thm} The first order derivatives of $w$ are
H\"older in $x$ uniformly in $(\tilde{c},\tilde{d}) \in Y$ and in $x
\in \bar{D}$, and the second order derivatives of $w$ are H\"older in
$x$ uniformly in $(\tilde{c},\tilde{d}) \in Y$ and locally uniformly
in $x \in D$.
\end{theorem}
\begin{proof}
By Theorem~\ref{globally-positive-existence}, for each
$(\tilde{c},\tilde{d}) \in Y$ there holds
\begin{equation*}
w((\tilde{c},\tilde{d})) = \frac{U_{(\tilde{c},\tilde{d}) \cdot
(-1)}(1,0) w((\tilde{c},\tilde{d}) \cdot (-1))}
{\norm{U_{(\tilde{c},\tilde{d}) \cdot (-1)}(1,0)
w((\tilde{c},\tilde{d}) \cdot (-1))}}.
\end{equation*}
It is a consequence of the continuity of $w$, the compactness of $Y$
and Proposition~\ref{joint-continuity-in-X-theta} that the
denominators on the right-hand side are positive and bounded away
from zero, uniformly in $Y$.  Now we apply the parabolic regularity
estimates~\cite[Theorem 5, Chapter 3]{Fri1}.
\end{proof}

\section{Principal spectrum and principal Lyapunov exponent}
\label{principal-spectrum}

In this section,  we collect the basic concepts and facts about the
principal spectrum and principal Lyapunov exponent of
\eqref{nonauton-eq1} and \eqref{random-eq1}.

\begin{definition}
\label{principal-spectrum-def}
In case of~\eqref{nonauton-eq1} we define its {\em principal
spectrum\/} to be the set of all limits
\begin{equation*}
\lim\limits_{n\to\infty} \frac{\ln\norm{U_{(c,d) \cdot
S_n}(T_n-S_n,0)w((c,d) \cdot S_n)}} {T_n-S_n},
\end{equation*}
where $T_n - S_n \to \infty$ as $n \to \infty$.
\end{definition}

The following proposition follows from the results contained in
\cite{JPSe} (cp., e.g., \cite[Thm.~2.10]{Mi3}).
\begin{theorem}
\label{principal-spectrum-thm2} The principal spectrum
of~\eqref{nonauton-eq1} is a compact interval $[\lambdainf(c,d)$,
$\lambdasup(c,d)]$.  Moreover, if $(c,d)$ is uniquely ergodic and
minimal then $\lambdainf(c,d) = \lambdasup(c,d)$.
\end{theorem}

\vskip2ex In the case of~\eqref{random-eq1}, for $\omega \in \Omega$
we write  $U_\omega(t,0)$ for $U_{(c^\omega,d^\omega)}(t,0)$ and
$w(\omega)$ for $w((c^{\omega},d^{\omega}))$.

\begin{theorem}
\label{principal-exponent-thm} For \eqref{random-eq1}, there exists
$\lambda (c,d)\in \RR$ such that
\begin{equation*}
\lambda (c,d)= \lim\limits_{T\to\infty}
\frac{\ln\|U_\omega(T,0)w(\omega)\|}{T}
\end{equation*}
for a.e.~$\omega \in \Omega$.
\end{theorem}
\begin{proof}
It follows from subadditive ergodic theorems (see \cite{Krengel}).
\end{proof}

\begin{definition}
\label{principal-exponent-def} The $\lambda(c,d)$ as in
Theorem~\ref{principal-exponent-thm} is called the {\em principal
Lyapunov exponent\/} of~\eqref{random-eq1}.
\end{definition}

\begin{remark}
In the existing literature, the principal spectrum is either defined
precisely as in Definition~\ref{principal-spectrum-def}
\textup{(}see~\cite{Mi3}\textup{)} or with the $L_2(D)$-norm replaced
by the norm in some fractional power space that embeds continuously
into $C^1(\bar{D})$ \textup{(}see, e.g.,~\cite{MiSh1}\textup{)}.  In
our setting, as $X_1$ is a one-dimensional invariant subbundle
spanned by a continuous function from $Y$ into $X$, we can replace
the $L_2(D)$-norm in Definition~\ref{principal-spectrum-def} with the
$X$-norm.
\end{remark}
\begin{remark}
Similarly, in the Definition~\ref{principal-exponent-def} the
$L_2(D)$-norm can be replaced with the $X$-norm.  Further,
in~\cite{MiSh1} the principal Lyapunov exponent was introduced as the
\textup{(}a.e.\ constant\textup{)} limit
\begin{equation*}
\lim\limits_{T\to\infty} \frac{\ln{\lVert U_{\omega}(T,0)}
\rVert_{X, X}}{T},
\end{equation*}
where $X$ is some fractional power space that embeds continuously
into $C^1(\bar{D})$.  With the help of \eqref{exp-sep-formula-1} one
can prove that for those $\omega \in \Omega$ for which $\lambda (c,d)
= \lim_{T\to\infty} \frac{\ln\norm{U_\omega(T,0)w(\omega)}}{T}$ there
holds also $\lambda (c,d) = \lim_{T\to\infty}
\frac{\ln{\norm{U_{\omega}(T,0)}}}{T}$ \textup{(}see ~the proof
of~\cite[Thm.~3.2(2)]{Mi1}\textup{)}.
\end{remark}
\begin{remark}
For the $L_2(D)$-theory of the principal spectrum and principal
Lyapunov exponents see the upcoming monograph~\cite{MiSh2}.
\end{remark}

We introduce now a useful concept.  For $(\tilde{c},\tilde{d}) \in Y$
put
\begin{align}
\label{introduction-kappa-eq} \kappa((\tilde{c},\tilde{d})) :=&
\int\limits_{D} \Bigl( \sum\limits_{i,j=1}^{N} a_{ij}(x) \frac{\p^2
w((\tilde{c},\tilde{d}))(x)}{\p x_{i} \p x_{j}} \Bigr)
w((\tilde{c},\tilde{d}))(x) \, dx \nonumber\\
&+ \int\limits_{D} \Bigl( \sum\limits_{i=1}^{N} a_{i}(x) \frac{\p
w((\tilde{c},\tilde{d}))(x)}{\p x_{i}} + \tilde{c}(0,x)
w((\tilde{c},\tilde{d}))(x) \Bigr) w((\tilde{c},\tilde{d}))(x) \,
dx.
\end{align}
By Proposition~\ref{evolution-op-in-smooth-case-prop1},
$w((\tilde{c},\tilde{d})) \in W^{2}_{2}(D)$, so
$\kappa((\tilde{c},\tilde{d}))$ is well defined.

The function $\kappa \colon Y \to \RR$ is continuous.  Indeed, notice
that applying integration by parts we can write
\begin{align}
\label{property-kappa-eq} \kappa((\tilde{c},\tilde{d})) = &-
\int\limits_{D} \sum\limits_{i,j=1}^{N} a_{ij}(x) \frac{\p
w((\tilde{c},\tilde{d}))(x)}{\p x_{i}} \frac{\p
w((\tilde{c},\tilde{d}))(x)}{\p x_{j}} \, dx \nonumber\\
&- \int\limits_{D} \sum\limits_{i=1}^{N} \Bigl(
\sum\limits_{j=1}^{N} \frac{\p a_{ij}(x)}{\p x_{i}} \frac{\p
w((\tilde{c},\tilde{d}))(x)}{\p x_j}\Bigr)
w((\tilde{c},\tilde{d}))(x) \, dx \nonumber\\
&+ \int\limits_{D} \Bigl( \sum\limits_{i=1}^{N} a_{i}(x) \frac{\p
w((\tilde{c},\tilde{d}))(x)}{\p x_{i}} + \tilde{c}(0,x)
w((\tilde{c},\tilde{d}))(x) \Bigr)
w((\tilde{c},\tilde{d}))(x) \, dx \nonumber\\
&+ \int\limits_{\p D} \sum\limits_{i=1}^{N} \Bigl(
\sum\limits_{j=1}^{N} a_{ij}(x) \frac{\p
w((\tilde{c},\tilde{d}))(x)}{\p x_{j}} \Bigr)
w((\tilde{c},\tilde{d}))(x) \nu_{i}(x) \, dS.
\end{align}
As $w \colon Y \to X$ is continuous, the above expression depends
continuously on $(\tilde{c},\tilde{d})$, too.

We point out that the function $\kappa((\tilde{c},\tilde{d}))$
introduced in \eqref{introduction-kappa-eq} is a very useful quantity
in the investigation of various properties of principal spectrum and
principal Lyapunov exponents.  This quantity will be heavily used in
next section.  In the rest of this section, we discuss how to use the
function $\kappa$ to characterize the principal spectrum and
principal Lyapunov exponents.

Let $\eta_{(\tilde{c},\tilde{d})}(t) :=
\norm{U_{(\tilde{c},\tilde{d})}(t,0)w((\tilde{c},\tilde{d}))}$ $( >
0)$. Then $\eta_{(\tilde{c},\tilde{d})}(t)$ is differentiable and
$U_{(\tilde{c},\tilde{d})}(t,0)w((\tilde{c},\tilde{d})) =
\eta_{(\tilde{c},\tilde{d})}(t)w((\tilde{c},\tilde{d}) \cdot t)$.
Hence $w((\tilde{c},\tilde{d}) \cdot t)$ is also differentiable in
$t$.  By \eqref{nonauton-eq2}, we have
\begin{align*}
&\dot{\eta}_{(\tilde{c},\tilde{d})}(t) w((\tilde{c},\tilde{d}) \cdot
t) + \eta_{(\tilde{c},\tilde{d})}(t)
\frac{\p}{\p t}w((\tilde{c},\tilde{d}) \cdot t) \\
=& \sum\limits_{i,j=1}^{N}
a_{ij}(x)\eta_{(\tilde{c},\tilde{d})}(t)\frac{\p ^2
w((\tilde{c},\tilde{d}) \cdot t)}{\p x_i\p
x_j}+\sum\limits_{i=1}^{N} a_{i}(x)\eta_{(\tilde{c},\tilde{d})}(t)
\frac{\p w((\tilde{c},\tilde{d}) \cdot t)}{\p x_i}\\
 &+ \tilde{c}(t,x)\eta_{(\tilde{c},\tilde{d})}(t)
w((\tilde{c},\tilde{d}) \cdot t).
\end{align*}
Taking the inner product of the above equation with
$w((\tilde{c},\tilde{d}) \cdot t)$ and observing that $\langle
w((\tilde{c},\tilde{d}) \cdot t), w((\tilde{c},\tilde{d}) \cdot t)
\rangle \equiv 1$ and $\langle \frac{\p}{\p t}w((\tilde{c},\tilde{d})
\cdot t), w((\tilde{c},\tilde{d}) \cdot t) \rangle \equiv 0$ we get
$\dot{\eta}_{(\tilde{c},\tilde{d})}(t) = \kappa((\tilde{c},\tilde{d})
\cdot t) \eta_{(\tilde{c},\tilde{d})}(t)$, that is,
\begin{equation}
\label{kappa-exponent} \frac{d}{dt}
\norm{U_{(\tilde{c},\tilde{d})}(t,0) w((\tilde{c},\tilde{d}))} =
\kappa((\tilde{c},\tilde{d}) \cdot t)
\norm{U_{(\tilde{c},\tilde{d})}(t,0) w((\tilde{c},\tilde{d}))}
\end{equation}
for any $(\tilde{c},\tilde{d}) \in Y$ and any $t \ge 0$.

By \eqref{kappa-exponent}, we have
\begin{equation}
\label{evolution-kappa}
\ln\norm{U_{(\tilde c,\tilde d) \cdot S}(T-S,0)
w((\tilde{c},\tilde{d})\cdot S)} = \int_S^ T
\kappa((\tilde{c},\tilde{d}) \cdot t) \,dt
\end{equation}
for any $(\tilde{c},\tilde{d}) \in Y$ and $S < T$. Then following
from Definition \ref{principal-spectrum-def} we have
\begin{theorem}
\label{principal-spectrum-thm3}
Let $[\lambdainf(c,d),\lambdasup(c,d)]$ be the principal spectrum
interval of \eqref{nonauton-eq1}.  Then
\begin{equation}
\label{minimum-spectrum}
\lambdainf(c,d) = \liminf_{T-S\to\infty} \frac{1}{T-S}\int_S^T
\kappa((c,d) \cdot t)\,dt
\end{equation}
and
\begin{equation}
\label{maximum-spectrum}
\lambdasup(c,d) = \limsup_{T-S\to\infty} \frac{1}{T-S}\int_S^T
\kappa((c,d) \cdot t)\,dt.
\end{equation}
\end{theorem}

\vskip2ex In the case of~\eqref{random-eq1} we write $\kappa(\omega)$
instead of $\kappa((c^{\omega},d^{\omega}))$.  We have
\begin{theorem}
\label{kappa-exponent-hm}
Consider~\eqref{random-eq1}.  Then
\begin{equation*}
\lambda = \lim\limits_{T\to \infty}\frac{1}{T}\int_0^T
\kappa(\theta_t\omega) \,dt = \int_{\Omega} \kappa(\cdot)
\,d\PP(\cdot)
\end{equation*}
for a.e.~$\omega \in \Omega$.
\end{theorem}
\begin{proof}
By the arguments of Lemma 3.4 in \cite{MiSh1}, the map $[\,\Omega \ni
\omega \mapsto (c^{\omega},d^{\omega}) \in Y\,]$ is measurable. The
theorem is then  a consequence of
Theorem~\ref{principal-exponent-thm}, Eq.~\eqref{evolution-kappa} and
Birkhoff's Ergodic Theorem (Lemma~\ref{ch5-pre-ergodic-lm}).
\end{proof}

We remark that if $c(t,x)$ and $d(t,x)$ are independent of $t$, then
$(c,d) = (\hat{c},\hat{d})$ and $\lambdainf(c,d) = \lambdasup(c,d) =
\lambda(c,d)$. Moreover, we have the following easy theorem about the
continuous dependence of $\lambda(c,d)$ on $(c,d)$.

\begin{theorem}
\label{continuous-dependence-of-principal-eigenvalue-thm}
If $c^{(n)}$ converges in $C(\bar{D})$ to $c$ and $d^{(n)}$ converges
in $C(\p D)$ to $d$ then $\lambda(c^{(n)},d^{(n)}) \to \lambda(c,d)$.
\end{theorem}

\section{Time averaging}
\label{main-results}

In this section we state and prove our results on the influence of
time variations on principal spectrum and principal Lyapunov exponent
of \eqref{nonauton-eq1} and \eqref{random-eq1}.

Consider \eqref{nonauton-eq1}. Let $\Sigma(c,d) :=
[\lambdainf(c,d),\lambdasup(c,d)]$ be the principal spectrum interval
of \eqref{nonauton-eq1}.  For $(\hat{c},\hat{d}) \in \hat{Y}(c,d)$
let $\lambda(\hat{c},\hat{d})$ denote the principal eigenvalue of an
averaged equation \eqref{nonauton-or-random-avg}. Recall that
\begin{align*}
\hat{Y}(c,d) = \{\,(\hat{c},\hat{d}):\, & \exists\, S_n < T_n \
\text{with} \ T_n-S_n \to\infty
\ \text{such that} \nonumber \\
& \ (\hat{c},\hat{d}) = \lim_{n\to\infty}
(\bar{c}(\cdot;S_n,T_n),\bar{d}(\cdot;S_n,T_n))\,\},
\end{align*}
where $\bar{c}(x;S_n,T_n) :=
\frac{1}{T_n-S_n}\int_{S_n}^{T_n}c(t,x)\,dt$, $\bar{d}(x;S_n,T_n) :=
\frac{1}{T_n-S_n}\int_{S_n}^{T_n}d(t,x)\,dt$, and the convergence is
in $C(\bar{D}) \times C(\p D)$.

Consider \eqref{random-eq1}.  Let $\lambda(c,d)$ be the principal
Lyapunov exponent.  Let
\begin{equation*}
\hat{c}(x) := \int_{\Omega} c(\omega,x)\,d\PP(\omega),\quad
\hat{d}(x) = \int_{\Omega} d(\omega,x)\,d\PP(\omega).
\end{equation*}
Let  $\lambda(\hat{c},\hat{d})$ be the principal eigenvalue of the
averaged equation~\eqref{nonauton-or-random-avg}.

Then we have
\begin{theorem}
\label{ch5-smoothlb-thm1}
\begin{itemize}
\item[{\rm (1)}]
Consider \eqref{nonauton-eq1}.  There is $(\hat{c},\hat{d}) \in
\hat{Y}(c,d)$ such that $\lambdainf(c,d) \ge
\lambda(\hat{c},\hat{d})$ and $\lambdasup(c,d) \ge
\lambda(\hat{c},\hat{d})$ for any $(\hat{c},\hat{d}) \in
\hat{Y}(c,d)$.

\item[{\rm (2)}]
Consider \eqref{random-eq1}.  $\lambda(c,d) \ge
\hat{\lambda}(\hat c,\hat d)$.
\end{itemize}
\end{theorem}

\begin{theorem}
\label{ch5-smoothlb-thm2}
\begin{itemize}
\item[{\rm (1)}]
Consider \eqref{nonauton-eq1}.  If $(c,d)$ is uniquely ergodic
and minimal, then $\lambdainf(c,d) = \lambdasup(c,d)$ and
$\lambdainf(c,d) = \lambda(\hat{c},\hat{d})$ for
$(\hat{c},\hat{d}) \in \hat{Y}(\hat{c},\hat{d})$ $(\hat{Y}(c,d)$
is necessarily a singleton$)$ if and only if $c(t,x) = c_{1}(x) +
c_{2}(t)$ and $d(t,x) = d(x)$.
\item[{\rm (2)}]
Consider \eqref{random-eq1}.  $\lambda(c,d) = \hat{\lambda}(c,d)$
if and only if there is $\Omega^{*} \subset \Omega$ with
$\PP(\Omega^{*}) = 1$ such that $c(\theta_{t}\omega,x) = c_{1}(x)
+ c_{2}(\theta_{t}\omega)$ for any $\omega \in \Omega^{*}$, $t
\in \RR$ and $x \in \bar{D}$, and $d(\theta_{t}\omega,x) = d(x)$
for any $\omega \in \Omega^{*}$, $t \in \RR$ and $x \in \p D$.
\end{itemize}
\end{theorem}

In the case that the boundary condition is of the Dirichlet or
Neumann type or of the Robin type with $d$ independent of $t$, the
above theorems have been proved in \cite{MiSh1}.  For completeness,
we will provide proofs of the theorems including the case that the
boundary condition is of the Robin type with $d$ depending on $t$. We
note that the proof in the following for Theorem
\ref{ch5-smoothlb-thm1} is not the same as that in \cite{MiSh1} even
in the case $d$ is independent of $t$.
\begin{proof}[Proof of Theorem \ref{ch5-smoothlb-thm1}]
First of all, let $(\tilde{c},\tilde{d}) = (c,d)$ in the case of
\eqref{nonauton-eq1} and $(\tilde{c},\tilde{d})=
(c^\omega,d^{\omega})$ in the case of \eqref{random-eq1} for some
given $\omega \in \Omega$.  For given $S$ and $T > 0$,  let
\begin{equation*}
\eta(t;\tilde{c},\tilde{d}, S) := \norm{U_{(\tilde{c},\tilde{d})
\cdot S}(t,0)w((\tilde{c},\tilde{d}) \cdot S)}, \qquad t \ge 0,
\end{equation*}
and
\begin{equation*}
\hat{w}(x;\tilde{c},\tilde{d}, S, T) :=
\exp\Bigl(\frac{1}{T}\int_0^T \ln{w((\tilde{c},\tilde{d}) \cdot
(t+S))(x)}\,dt\Bigr)
\end{equation*}
for $x\in D$ and
\begin{equation*}
\hat w(x;\tilde{c},\tilde{d},S,T)=0
\end{equation*}
for $x\in\p D$ in the Dirichlet boundary condition case, and
\begin{equation*}
\hat{w}(x;\tilde{c},\tilde{d}, S, T) :=
\exp\Bigl(\frac{1}{T}\int_0^T \ln{w((\tilde{c},\tilde{d}) \cdot
(t+S))(x)}\,dt\Bigr)
\end{equation*}
for $x\in \bar D$ in the Neumann and Robin boundary conditions cases.
Note that $\hat w(x;\tilde c,\tilde d, S,T)\in C(\bar D)$.

 Let $\bar{v}(t,x;\tilde{c},\tilde{d},S) :=
w((\tilde{c},\tilde{d}) \cdot (t+S))(x)$. We have that
$\eta(t;\tilde{c},\tilde{d}, S)$ satisfies
\begin{equation}
\label{eta-nonauton}
\eta_t(t;\tilde{c},\tilde{d}, S) = \kappa((\tilde{c},\tilde{d})
\cdot(t+S)) \eta(t;\tilde{c},\tilde{d}, S),
\end{equation}
and  $\bar{v}(t,x;\tilde{c},\tilde{d}, S)$ satisfies
\begin{equation}
\label{nonauton-hat-v-eq1}
\begin{cases}
\disp\frac{\p \bar{v}}{\p t} = \sum_{i,j=1}^{N} a_{ij}(x)\frac{\p^2
\bar{v}}{\p x_{i} \p x_{j}} + \sum_{i=1}^{N} a_i(x) \frac{\p
\bar{v}}{\p x_i} \\[2ex]
\qquad\quad {} + \tilde{c}(t+S,x)\bar{v} -
\kappa((\tilde{c},\tilde{d}) \cdot (t+S)) \bar{v}, &
\quad x \in D \\[2ex]
 \mathcal{\tilde B}(t+S) \bar{v} = 0, & \quad x \in \p D,
\end{cases}
\end{equation}
where $\mathcal{\tilde B}(\cdot)$ is as in \eqref{nonauton-eq2}.
 Theorem
\ref{Holder-continuity-thm} allows us to differentiate sufficiently
many times to obtain that for any $x \in {D}$  ($x$ can also be in
$\bar D$ in the Neumann and Robin boundary conditions cases) we have
\begin{align}
&\label{nonauton-hat-w-eq1}
\frac{\p\hat{w}}{\p x_i}(x;\tilde{c},\tilde{d},S,T) \nonumber\\
= \null & \hat{w}(x;\tilde{c},\tilde{d}, S,T) \frac{1}{T}\int_0^T
\Bigl(\frac{1}{w((\tilde{c},\tilde{d}) \cdot (t+S))(x)} \frac{\p
w((\tilde{c},\tilde{d}) \cdot (t+S))(x)}{\p x_i}\Bigr)\,dt,
\end{align}
and that for any $x \in D$ we have
\begin{align}
& \label{nonauton-hat-w-eq2} \frac{\p^2\hat{w}}{\p x_i\p x_j} \nonumber\\
= \null &
\hat{w}(x;\tilde{c},\tilde{d}, S,T) \Bigg( \frac{1}{T^2}\int_{0}^{T}
\Bigl(\frac{1}{w((\tilde{c},\tilde{d}) \cdot (t+S))(x)} \frac{\p
w((\tilde{c},\tilde{d}) \cdot (t+S))(x)}{\p
x_i}\Bigr)\,dt \cdot \nonumber\\
&\qquad\qquad\qquad\qquad\qquad\int_{0}^{T} \Bigl(
\frac{1}{w((\tilde{c},\tilde{d}) \cdot (t+S))(x)}\frac{\p
w((\tilde{c},\tilde{d}) \cdot (t+S))(x)}{\p x_j}\Bigr)\,dt\Bigg ) \nonumber\\
\null &+ \hat{w}(x;\tilde{c},\tilde{d}, S,T) \frac{1}{T}
\int_0^T\Bigl(\frac{1}{w((\tilde{c},\tilde{d}) \cdot (t+S))(x)}
\frac{\p ^2 w((\tilde{c},\tilde{d}) \cdot (t+S))(x)}{\p x_i\p
x_j} \nonumber\\
\null &- \frac{1}{w^2((\tilde{c},\tilde{d}) \cdot (t+S))(x)} \frac{\p
w((\tilde{c},\tilde{d}) \cdot (t+S))(x)}{\p x_i} \frac{\p
w((\tilde{c},\tilde{d}) \cdot (t+S))(x)}{\p x_j}\Bigr)\,dt.
\end{align}
Then by \eqref{nonauton-hat-v-eq1},  $\hat{w} =
\hat{w}(x;\tilde{c},\tilde{d},S,T)$ satisfies
\begin{align}
\label{nonauton-hat-w-eq3} &\sum_{i,j=1}^{N}  a_{ij}(x) \frac{\p^2
\hat{w}}{\p x_{i} \p x_{j}} + \sum_{i=1}^{N} a_i(x) \frac{\p
\hat{w}}{\p x_i}
\nonumber \\
= \null & \Bigl(\frac{1}{T}\int\limits_{0}^{T} \frac{1}{\bar{v}} \frac{\p
\bar{v}}{\p t}(t,x;\tilde{c},\tilde{d},S)\,dt\Bigr ) \hat{w} \nonumber \\
& \quad + \Bigl(\frac{1}{T} \int_0^T \kappa((\tilde{c},\tilde{d})
\cdot (t+S))\,dt -
\frac{1}{T} \int_0^T \tilde{c}(t+S,x)\,dt\Bigr) \hat{w}\nonumber \\
& \quad + \hat{w}\sum_{i,j=1}^N a_{ij}(x) \Biggl(
\frac{1}{T}\int_0^T \Bigl(\frac{1}{ w((\tilde{c},\tilde{d}) \cdot
(t+S))} \frac{\p w((\tilde{c},\tilde{d})
\cdot (t+S))}{\p x_i} \Bigr)\,dt \cdot  \nonumber \\
&\quad\quad  \frac{1}{T} \int_0^T \Bigl(\frac
{1}{w((\tilde{c},\tilde{d}) \cdot (t+S))}\frac{\p w(\tilde{a} \cdot
(t+S))}{\p x_j} \Bigr)
\,dt \Biggr ) \\
& \quad - \hat{w} \sum_{i,j=1}^N a_{ij}(x) \cdot
\nonumber \\
& \quad\quad\quad\frac{1}{T}\int_0^T
\Bigl(\frac{1}{w^2((\tilde{c},\tilde{d}) \cdot (t+S))} \frac{\p
w((\tilde{c},\tilde{d}) \cdot (t+S))}{\p x_i} \frac{\p
w((\tilde{c},\tilde{d}) \cdot (t+S))}{\p x_j} \Bigr)\,dt\nonumber
\end{align}
for $x \in D$, and $\mathcal{\hat{B}}_{S,T} \hat{w} = 0$ for $x \in
\p D$, where
\begin{equation}
\label{nonauton-hat-w-eq4}
\mathcal{\hat{B}}_{S,T}\hat{w} :=
\begin{cases}
\hat{w} & \text{(Dirichlet)} \\[2ex]
\disp \sum_{i=1}^N b_i(x)\frac{\p \hat{w}}{\p x_i} &
\text{(Neumann)} \\[2ex]
\disp \sum_{i=1}^N b_i(x)\frac{\p \hat{w}}{\p x_i} +
\left(\frac{1}{T}\int_0^T \tilde{d}(t+S,x)\,dt \right) \hat{w} &
\text{(Robin)}.
\end{cases}
\end{equation}
By Lemma~\ref{ch5-pre-holder-lm}(1),
\begin{align}
\label{nonauton-hat-w-eq5} &\sum_{i,j=1}^{N}  a_{ij}(x) \frac{\p^2
\hat{w}}{\p x_{i} \p x_{j}} + \sum_{i=1}^{N} a_i(x) \frac{\p
\hat{w}}{\p x_i} \nonumber \\
\le \null & \Bigl(\frac{1}{T}\int\limits_{0}^{T} \frac{1}{\bar{v}}
\frac{\p \bar{v}}{\p t}(t,x;\tilde{c},\tilde{d},S)\,dt \Bigr )
\hat{w}
\nonumber \\
& + \Bigl(\frac{1}{T}\int_0^T \kappa((\tilde{c},\tilde{d}) \cdot
(t+S))\,dt - \frac{1}{T}\int_{0}^{T} \tilde{c}(t+S,x)\,dt \Bigr)
\hat{w}.
\end{align}

Note that $\bar{v}(t,x;\tilde{c},\tilde{d},S) =
w((\tilde{c},\tilde{d}) \cdot (t+S))(x)$ and by
Theorem~\ref{pre-bounds-of-w-thm}, for a fixed compact $D_0 \Subset
D$ there are $0 < m(D_0) < M$ such that $m(D_0) \le
\bar{v}(t,x;\tilde{c},\tilde{d}, S) \le M$ for any
$(\tilde{c},\tilde{d}) \in Y$, $t,S \in \RR$, and $x \in D_0$.  Hence
\begin{align}
&\label{nonauton-bar-v-eq3} \lim\limits_{T\to\infty} \frac{1}{T}
\int\limits_{0}^{T}
\frac{1}{\bar{v}} \frac{\p \bar{v}}{\p t}(t,x;\tilde{c},\tilde{d}, S)\,dt \nonumber\\
=& \lim\limits_{T\to\infty} \frac{1}{T}
(\ln\bar{v}(T,x;\tilde{c},\tilde{d},S) -
\ln\bar{v}(0,x;\tilde{c},\tilde{d},S)) = 0
\end{align}
for any $(\tilde{c},\tilde{d}) \in Y$, $S \in \RR$, and $x \in D$.
Moreover, the limits are uniform in $(x,S) \in D_0 \times \RR$ for
any compact $D_0 \Subset D$.

\medskip

\noindent (1) We first prove that $\lambdainf(c,d) \ge
\lambda(\hat{c},\hat{d})$ for some $(\hat{c},\hat{d}) \in
\hat{Y}(c,d)$.

Note that for given $S$, $T > 0$,
\begin{equation*}
\eta(t;c,d, S) = \norm{U_{(c,d) \cdot S}(t,0)w((c,d) \cdot S)},
\end{equation*}
\begin{equation*}
\hat{w}(x;c,d, S,T) = \exp\Bigl(\frac{1}{T}\int_S^{T+S}\ln{w((c,d)
\cdot t)(x)}\,dt\Bigr),
\end{equation*}
\begin{equation*}
\frac{1}{T} \int_0^T \kappa((c,d) \cdot (t+S)) \,dt =
\frac{1}{T}\int_S^{T+S}\kappa((c,d) \cdot t)dt,
\end{equation*}
and
\begin{equation*}
\frac{1}{T} \int_{0}^{T} c(t+S,x)\,dt = \frac{1}{T}
\int_S^{T+S}c(t,x)dt.
\end{equation*}

By Theorem \ref{principal-spectrum-thm3} there are $(S_n)$, $(T_n)$
with $T_n \to \infty$ such that
\begin{equation*}
\frac{1}{T_n} \int_{S_n}^{T_n+S_n}\kappa((c,d) \cdot t)\,dt =
\frac{\ln{\eta(T_n;c,d,S_n)}}{T_n} \to \lambdainf(c,d).
\end{equation*}
Without loss of generality we may assume that the limits\break
$\lim_{n\to\infty} \frac{1}{T_n} \int_{S_n}^{T_n+S_n}c(t,x) \,dt$ and
$\lim_{n\to\infty}\frac{1}{T_n}\int_{S_n}^{T_n+S_n}d(t,x)\,dt$ exist,
uniformly in $x \in \bar{D}$ (resp.\ in $x \in \p D$). Denote these
limits by $(\hat{c},\hat{d})$.

In the Dirichlet case, it is a consequence of
Theorems~\ref{pre-bounds-of-w-thm} and~\ref{Holder-continuity-thm}
that for each compact $D_0 \Subset D$ the sets
$\{\,\hat{w}(\cdot;c,d,S_n,T_n)|_{D_0}: n = 1, 2, \dots\,\}$,
$\{\,(\p \hat{w}/\p x_i)(\cdot;c,d$, $S_n$, $T_n)|_{D_0}: n = 1, 2,
\dots\,\}$ ($i = 1, \dots, N$) and $\{\,(\p^2 \hat{w}/\p x_{i}\p
x_{j})(\cdot;c,d,S_n,T_n)|_{D_0}: n = 1, 2, \dots\,\}$ ($i, j = 1,
\dots, N$) have compact closures in $C(D_0)$.

In the Neumann and Robin cases it is a consequence of
Theorems~\ref{pre-bounds-of-w-thm} and~\ref{Holder-continuity-thm}
that the sets $\{\,\hat{w}(\cdot;c,d,S_n,T_n): n = 1, 2, \dots\,\}$
and $\{\,(\p \hat{w}/\p x_i)(\cdot;c,d,S_n,T_n): n = 1, 2, \dots\,\}$
($i = 1, \dots, N$) have compact closures in $C(\bar{D})$, and that
for each compact $D_0 \Subset D$ the sets $\{\,(\p^2 \hat{w}/\p
x_{i}\p x_{j})(\cdot;c,d,S_n,T_n)|_{D_0}: n = 1, 2, \dots\,\}$ ($i, j
= 1, \dots, N$) have compact closures in $C(D_0)$.

We may thus assume that there is $w^{*} = w^*(x)$ such that
\begin{equation}
\label{limit1}
\lim_{n\to\infty} \hat{w}(x;c,d,S_n,T_n) = w^*(x)
\end{equation}
\begin{equation}
\label{limit2}
\lim_{n\to\infty} \frac{\p \hat{w}(x;c,d,S_n,T_n)}{\p x_i} =
\frac{\p w^*(x)}{\p x_i}
\end{equation}
\begin{equation}
\label{limit3}
\lim_{n\to\infty}\frac{\p^2 \hat{w}(x;c,d,S_n,T_n)}{\p x_i\p x_j} =
\frac{\p^2 w^*(x)}{\p x_i\p x_j}
\end{equation}
for $i, j = 1,2, \dots, N$ and $x \in D$.  In the Dirichlet boundary
conditions case, it follows from Theorems~\ref{pre-bounds-of-w-thm}
and~\ref{Holder-continuity-thm} that $w^{*}$ can be extended to a
function continuous on $\bar{D}$ by putting $w^{*}(x) = 0$ for $x \in
\p D$. Moreover, by Theorem \ref{pre-bounds-of-w-thm}, $w^*(x)
> 0$ for $x \in D$.

Regarding the uniformity of convergence, in the Dirichlet case, the
limit in \eqref{limit1} is uniform for $x$ in $\bar D$ and the limits
in \eqref{limit2} and \eqref{limit3} are uniform for $x$ in any
compact subset $D_0 \Subset D$, and in the Neumann and Robin cases,
the limits in \eqref{limit1} and \eqref{limit2} are uniform for $x\in
\bar D$ and the limit \eqref{limit3} is uniform for $x$ in any
compact subset $D_0 \Subset D$.

We claim that $\lambdainf(c,d) \ge \lambda(\hat{c},\hat{d})$. In
fact, by \eqref{nonauton-hat-w-eq5}--\eqref{limit3},
\begin{equation*}
\begin{cases}
\disp \sum_{i,j=1}^{N} a_{ij}(x) \frac{\p^2 {w}^*}{\p x_{i} \p
x_{j}} + \sum_{i=1}^{N} a_i(x) \frac{\p {w}^*}{\p x_i} + (\hat{c}(x)
- \lambda_{\rm min}(c,d)){w}^* \leq 0, & x \in D,
\\[2.5ex]
\hat{\mathcal{B}}{w}^* = 0, & x \in \p D,
\end{cases}
\end{equation*}
where
\begin{equation*}
\mathcal{\hat{B}}w^* :=
\begin{cases}
w^{*} \qquad  & \text{(Dirichlet)} \\[1.5ex]
\disp \sum_{i=1}^{N} b_{i}(x) \frac{\p w^*}{\p x_{i}}\qquad  &
\text{(Neumann)}
\\[1.5ex]
\disp \sum_{i=1}^{N} b_{i}(x) \frac{\p w^*}{\p x_{i}} + \hat{d}(x)
w^* \qquad & \text{(Robin)}.
\end{cases}
\end{equation*}
This implies that $w(t,x) = {w}^*(x)$ is a supersolution of
\begin{equation}
\label{nonauton-eq3}
\begin{cases}
w_t = \disp \sum_{i,j=1}^{N} a_{ij}(x) \frac{\p^2 {w}}{\p x_{i} \p
x_{j}} + \sum_{i=1}^{N} a_i(x) \frac{\p {w}}{\p
x_i}\\\qquad\qquad {} + (\hat{c}(x) - \lambdainf(c,d)){w},
\quad & x \in D,
\\[2.5ex]
\hat{\mathcal{B}}{w} = 0, \quad & x \in \p D.
\end{cases}
\end{equation}
Let $w(t,x;\hat{w})$ be the solution of \eqref{nonauton-eq3} with
initial condition $w(0,x;\hat{w}) = {w}^*(x)$. Then we have
\begin{equation}
\label{useful-eq1} w(t,x;\hat{w}) \le {w}^*(x)
\end{equation}
for $x \in D$ and $t \ge 0$.  Note that $\lambda(\hat{c},\hat{d}) -
\lambdainf(x,d)$ is the principal eigenvalue of
~\eqref{nonauton-or-random-avg} with $(\hat{c},\hat{d})$ being
replaced by $(\hat{c} - \lambdainf(c,d),\hat{d})$.  It then follows
from \eqref{useful-eq1} together with the positivity of ${w}^*(x)$
that
\begin{equation}
\label{useful-eq2}
\lambda(\hat{c},\hat{d}) - \lambdainf(c,d) \le 0.
\end{equation}
This implies that
\begin{equation*}
\lambda(\hat{c},\hat{d}) \le \lambdainf(c,d).
\end{equation*}

Next, we prove $\lambdasup(c,d) \ge \lambda(\hat{c},\hat{d})$ for any
$(\hat{c},\hat{d}) \in \hat{Y}(c,d)$.  For any $(\hat{c},\hat{d}) \in
\hat{Y}(c,d)$ there are $(S_n), (T_n)$ with $T_n \to \infty$ such
that
\begin{equation*}
\frac{1}{T_n} \int_{S_n}^{T_n+S_n}c(t,x)\,dt \to \hat{c}(x)
\end{equation*}
and
\begin{equation*}
\frac{1}{T_n} \int_{S_n}^{T_n+S_n} d(t,x)\,dt \to \hat{d}(x),
\end{equation*}
uniformly in $x \in \bar{D}$ (resp.\ uniformly in $x \in \p D$).
Without loss of generality, assume that
\begin{equation*}
\frac{1}{T_n} \int_{S_n}^{T_n+S_n} \kappa((c,d) \cdot t)\,dt \to
\lambda_0.
\end{equation*}
By arguments similar to the above, $\lambda_0 \ge
\lambda(\hat{c},\hat{d})$. Note that $\lambdasup(c,d) \ge \lambda_0$.
Then we have $\lambdasup(c,d) \ge \lambda(\hat{c},\hat{d})$.

(2)  By Lemma~\ref{averaging-uniform}, there is $\Omega_1 \subset
\Omega$ with $\PP(\Omega_1) = 1$ such that
\begin{equation*}
\hat{c}(x) = \lim_{T\to\infty}\frac{1}{T}\int_0^{T}
c(\theta_t\omega,x)\,dt
\end{equation*}
for any $\omega \in \Omega_1$ and any $x \in \bar{D}$, uniformly in
$\bar{D}$, and
\begin{equation*}
\hat{d}(x) = \lim_{T\to\infty}\frac{1}{T}\int_0^{T}
d(\theta_{t}\omega,x)\,dt
\end{equation*}
for any $\omega \in \Omega_1$ and any $x \in \p D$, uniformly in $\p
D$.

By Theorem \ref{principal-exponent-thm}, there is $\Omega_2 \subset
\Omega$ with $\PP(\Omega_2)=1$ such that
\begin{equation*}
\lambda = \lim\limits_{T\to\infty}
\frac{\ln\norm{U_\omega(T,0)w(\omega)}}{T}
\end{equation*}
for any $\omega \in \Omega_2$.

Take an $\omega \in \Omega_1 \cap \Omega_2$.  Then for any $T_n \to
\infty$,
\begin{equation*}
\frac{1}{T_n}\int_0^{T_n} c(\theta_t\omega,x)\,dt =
\frac{1}{T_n}\int_0^{T_n}c^\omega(t,x)dt \to \hat{c}(x) \quad \text{
uniformly for } x \in \bar{D},
\end{equation*}
\begin{equation*}
\frac{1}{T_n}\int_0^{T_n} d(\theta_{t}\omega,x)\,dt =
\frac{1}{T_n}\int_0^{T_n}d^\omega(t,x)dt\to \hat{d}(x)\quad
\text{uniformly for } x \in \p D,
\end{equation*}
and
\begin{equation*}
\frac{\ln\eta(T_n;c^{\omega},d^{\omega},0)}{T_n} \to \lambda.
\end{equation*}
By arguments as in the proof of Part (1), we must have $\lambda \ge
\hat{\lambda}$.
\end{proof}

\begin{proof}[Proof of Theorem \ref{ch5-smoothlb-thm2}]
We first prove (2) for the reason that (2) will be used in the proof
of (1).

First,  suppose that $c(\omega,x) = c_{1}(x) +
c_{2}(\theta_{t}\omega)$ for any $x \in \bar{D}$, any $t \in \RR$ and
any $\omega \in \Omega^*$. Without loss of generality, we may assume
$\int_{\Omega} c_{2}(\omega) \,d\PP(\omega) = 0$ and $\PP(\Omega^*) =
1$ (for otherwise, we change $c_1(x)$ to
$c_1(x)+\int_{\Omega}c_2(\omega)\,d\PP(\omega)$ and change
$c_2(\omega)$ to $c_2(\omega) - \int_\Omega
c_2(\omega)\,d\PP(\omega)$).  Suppose also that
$d(\theta_{t}\omega,x) = d(x)$.  One has $\hat{c}(x) = c_{1}(x)$ for
$x \in \bar{D}$, and $\hat{d}(x) = d(x)$ for $x \in \p D$.  Let
$u(x)$ be the positive principal eigenfunction of
\eqref{nonauton-or-random-avg} normalized so that its $L_2(D)$-norm
equals $1$, and let
\begin{equation*}
v(t,x;\omega) := u(x)\exp{\Bigl(\hat{\lambda} t + \int_0^t
c_{2}(\theta_s\omega)\,ds\Bigr)}
\end{equation*}
for $t \in \RR$, $x \in \bar{D}$ and $\omega \in \Omega^*$.  It is
then not difficult to see that for any $\omega \in \Omega^{*}$ the
function $[\,\RR \ni t \mapsto v(t,\cdot;\omega) \in L_2(D)\,]$ is
the (necessarily unique) normalized globally positive solution
of~\eqref{random-eq1}.  For a.e.~$\omega \in \Omega$, $\lambda =
\lim_{t\to\infty}(1/t)\ln{\norm{v(t,\cdot;\omega)}}$.  It follows
with the help of Birkhoff's Ergodic Theorem
(Lemma~\ref{ch5-pre-ergodic-lm}) that the last term equals
$\hat{\lambda}$ for a.e.~$\omega \in \Omega^{*}$.  Consequently,
$\lambda = \hat{\lambda}$.

Conversely, let $\Omega_1$ and $\Omega_2$ be as in the proof of
Theorem \ref{ch5-smoothlb-thm1}(2).  We write $\eta(t;\omega)$ for
$\eta(t;c^{\omega},d^{\omega},0)$ and $\hat{w}(x;\omega,T)$ for
$\hat{w}(x;c^{\omega},d^{\omega},0,T)$, respectively.  Then
\begin{equation*}
\eta(t;\omega) = \norm{U_{\omega}(t,0)w(\omega)}
\end{equation*}
and
\begin{equation*}
\hat{w}(x;\omega,T) = \exp\Bigl(\frac{1}{T}\int_0^T \ln
w(\theta_t\omega)(x) \,dt \Bigr).
\end{equation*}

Let
\begin{equation*}
\phi(x) := \exp{\int\limits_{\Omega}\ln{w(\omega)(x)}\,d\PP(\omega)}
\quad\text{for}\quad x \in \bar{D}
\end{equation*}
in the case of Neumann or Robin boundary condition, and
\begin{equation*}
\phi(x) :=
\begin{cases}
\disp \exp{\int\limits_{\Omega}\ln{w(\omega)(x)}
\,d\PP(\omega)} \quad & \text{for}\quad x \in D
\\
0 \quad & \text{for}\quad x \in \p D
\end{cases}
\end{equation*}
in the case of  Dirichlet boundary condition.  By
Lemma~\ref{averaging-uniform}, there is $\Omega_3 \subset \Omega$
with $\PP(\Omega_3) = 1$ such that
\begin{equation}
\label{random-phi-eq1} \phi(x) = \lim_{T\to \infty}
\exp\Bigl(\frac{1}{T} \int_0^{T} \ln{w(\theta_{t}\omega)(x)}\,dt
\Bigr) = \lim_{T\to\infty} \hat{w}(x;\omega,T)
\end{equation}
for any $\omega \in \Omega_3$ and $x \in D$.  Clearly, $\phi(x)
> 0$ for $x \in D$.

Observe that by Theorems \ref{pre-bounds-of-w-thm} and
\ref{Holder-continuity-thm}, $\disp \frac{\p w(\omega)(x)}{\p x_i}$
($i = 1,2, \dots, N$) ($\disp\frac{\p ^2 w(\omega)(x)}{\p x_i\p
x_j}$, $i,j = 1,2,\dots,N$) are locally H\"older continuous in $x \in
\bar{D}$ ($x \in D$) uniformly in $\omega \in \Omega$, and for a
fixed $x \in \bar{D}$ ($x \in D$) they are bounded in $\omega \in
\Omega$.  Hence, Lemma~\ref{averaging-derivative} together with
Eqs.~\eqref{nonauton-hat-w-eq1} and \eqref{nonauton-hat-w-eq2} gives
us the existence of $\Omega_4 \subset \Omega$ with $\PP(\Omega_4) =
1$ such that
\begin{align}
\label{random-phi-eq2}
\frac{\p\phi}{\p x_i}(x)&
= \lim_{T\to\infty}\frac{\p \hat{w}(x;\omega,T)}{\p x_i} \nonumber \\
& = \phi(x)\int_\Omega \Bigl( \frac{1}{w(\cdot)(x)}\frac{\p
w(\cdot)(x)}{\p x_i}\Bigr)\,d\PP(\cdot),
\end{align}
\begin{align}
\label{random-phi-eq3}
\frac{\p^2\phi}{\p x_i\p x_j}&
= \lim_{T\to\infty}\frac{\p^2 \hat{w}(x;\omega,T)}{\p x_i\p x_j} \nonumber \\
& = \phi(x)\int_\Omega \Bigl(\frac{1}{w(\cdot)(x)} \frac{\p
w(\cdot)(x)}{\p x_i}\Bigr)\,d\PP(\cdot)
\int_\Omega\Bigl(\frac{1}{w(\cdot)(x)} \frac{\p w(\cdot)(x)}{\p x_j}
\Bigr)\,d\PP(\cdot)
\nonumber\\
&\quad + \phi(x) \int_\Omega\Bigl(\frac{1}{w(\cdot)(x)} \frac{\p ^2
w(\cdot)(x)}{\p x_i\p x_j} - \frac{1}{w^2(\cdot)(x)} \frac{\p
w(\cdot)(x)} {\p x_i} \frac{\p w(\cdot)(x)}{\p
x_j}\Bigr)\,d\PP(\cdot)
\end{align}
and
\begin{equation*}
\lim_{T\to\infty}\frac{1}{T}\int_0^T
\frac{1}{w(\theta_t\omega)(x)}\frac{\p w(\theta_t\omega)(x)}{\p
x_i}\,dt = \int_\Omega\frac{1}{w(\omega)(x)}\frac{\p
w(\omega)(x)}{\p x_i} \,d\PP(\cdot),
\end{equation*}
\begin{align*}
&\lim_{T\to\infty}\frac{1}{T}\int_0^T\frac{1}{w^2(\theta_t\omega)(x)}
\frac{\p w(\theta_t(\omega)(x)}{\p x_i} \frac{\p
w(\theta_t\omega)(x)}{\p x_j}\,dt \\
=& \int_{\Omega}\frac{1}{w(\omega)(x)}\frac{\p w(\omega)(x)}{\p
x_i}\frac{\p w(\omega)(x)}{\p x_j}\, d\PP(\cdot)
\end{align*}
for $\omega \in \Omega_4$, $x \in D$, and
\begin{equation*}
\mathcal{\hat{B}} \phi = 0
\end{equation*}
for $\omega \in \Omega_4$, $x \in \p D$, where
\begin{equation}
\label{random-phi-eq4}
\mathcal{\hat{B}}\phi :=
\begin{cases}
\phi & \text{(Dirichlet)} \\[1.5ex]
\disp \sum_{i=1}^{N} b_{i}(x) \frac{\p \phi}{\p x_{i}} &
\text{(Neumann)}
\\[1.5ex]
\disp \sum_{i=1}^{N} b_{i}(x) \frac{\p \phi}{\p x_{i}} +
\hat{d}(x)\phi & \text{(Robin)}.
\end{cases}
\end{equation}

Let $\Omega_0 := \Omega_1 \cap \Omega_2 \cap \Omega_3 \cap \Omega_4$.
Then \eqref{random-phi-eq1}--\eqref{random-phi-eq4} hold for any
$\omega \in \Omega_0$.

Put $\bar{v}(t,x;\omega) := w(\theta_t\omega)(x)$.  Since, by
Theorem~\ref{pre-bounds-of-w-thm}, for a fixed $x \in D$ there are $0
< m < M$ such that  $m \le w(\omega)(x) \le M$ for any $\omega \in
\Omega$, we have
\begin{equation}
\label{random-bar-v-eq3}
\lim\limits_{T\to\infty} \frac{1}{T} \int\limits_{0}^{T}
\frac{1}{\bar{v}} \frac{\p \bar{v}}{\p t}(t,x;\omega)\,ds =
\lim\limits_{T\to\infty} \frac{1}{T} (\ln{w(\theta_{T}\omega)(x)} -
\ln{w(\omega)(x)}) = 0
\end{equation}
for any $\omega \in \Omega_0$ and $x \in D$.

Consequently, by \eqref{nonauton-hat-w-eq3} and
\eqref{nonauton-hat-w-eq4} we have
\begin{align}
\label{ch5-smooth-phi-eq7} &\sum_{i,j=1}^N  a_{ij}(x)
\frac{\p^{2}\phi}{\p x_{i} \p x_{j}} + \sum_{i=1}^{N} a_{i}(x)
\frac{\p \phi}{\p x_i}
\nonumber \\
 =& (\lambda - \hat{c}(x))\phi\nonumber\\
&  + \phi\sum_{i,j=1}^N a_{ij}(x) \int_\Omega
\Bigl(\frac{1}{w(\cdot)} \frac{\p w(\cdot)}{\p
x_i}\Bigr)\,d\PP(\cdot) \int_\Omega \Bigl(\frac
{1}{w(\cdot)}\frac{\p w(\cdot)}{\p x_j}\Bigr) \,d\PP(\cdot)
\nonumber\\
&  - \phi \sum_{i,j=1}^N a_{ij}(x)
\int_\Omega\Bigl(\frac{1}{w^2(\cdot)} \frac{\p w(\cdot)}{\p x_i}
\frac{\p w(\cdot)}{\p x_j}\Bigr)\,d\PP(\cdot)
\end{align}
for $x \in D$, and
\begin{equation*}
\hat{\mathcal{B}}\phi = 0 \quad \text{for} \quad x \in \p D.
\end{equation*}

Suppose that $\lambda = \hat{\lambda}$. Consider
\begin{equation}
\label{additional-eq}
\begin{cases}
\disp u_t = \sum_{i,j=1}^N a_{ij}(x)\frac{\p^2 u}{\p x_i\p x_j} +
\sum_{i=1}^Na_i(x) \frac{\p
u}{\p x_i} + (\hat{c}(x) - \lambda)u,\quad x\in D, \\
\\
\hat{\mathcal{B}}u = 0.
\end{cases}
\end{equation}

\noindent We have that $0$ is the principal eigenvalue of
\eqref{additional-eq}.  Let $\hat{\phi}$ be a positive principal
eigenfunction of \eqref{additional-eq}.  Let $u(t,x;\phi)$ be the
solution of \eqref{additional-eq}  with initial condition
$u(0,x;\phi) = \phi(x)$.  By Lemma~\ref{ch5-pre-holder-lm}(2),
\begin{align*}
&\quad\sum_{i,j=1}^N a_{ij}(x) \int_\Omega \Bigl(\frac{1}{w(\cdot)}
\frac{\p w(\cdot)}{\p x_i}\Bigr)\,d\PP(\cdot) \int_\Omega
\Bigl(\frac {1}{w(\cdot)}\frac{\p w(\cdot)}{\p x_j}\Bigr)
\,d\PP(\cdot)\\
& - \sum_{i,j=1}^N a_{ij}(x) \int_\Omega\Bigl(\frac{1}{w^2(\cdot)}
\frac{\p w(\cdot)}{\p x_i} \frac{\p w(\cdot)}{\p
x_j}\Bigr)\,d\PP(\cdot) \le 0
\end{align*}
for all $x \in D$.  This together with \eqref{ch5-smooth-phi-eq7}
implies that $\phi(x)$ is a supersolution of \eqref{additional-eq}
and hence
\begin{equation}
\label{additional-eq1}
u(t,x;\phi) \le \phi(x) \quad \text{ for } \quad x \in D, \quad t
\ge 0.
\end{equation}
We apply now Theorem~\ref{globally-positive-existence} to the
autonomous problem~\eqref{additional-eq}.  In this case, $Y$ is a
singleton, $w = \hat{\phi}$, and $w^{*}(x) > 0$ for a.e.~$x \in D$.
It follows then that $\langle \phi, w^{*} \rangle > 0$ and $\langle
\hat{\phi}, w^{*} \rangle > 0$.  By taking $\alpha := \langle \phi,
w^{*} \rangle / \langle \hat{\phi}, w^{*} \rangle$ ($ > 0$) we see
that
\begin{equation*}
\phi = \alpha\hat{\phi} + \hat{\psi},
\end{equation*}
where $\hat \psi \in X$ is such that $\langle \hat \psi, w^{*}
\rangle = 0$. Note that $u(t,x;\phi) = \alpha \hat{\phi}(x) +
u(t,x;\hat{\psi})$, where $u(t,x;\hat \psi)$ is the solution of
\eqref{additional-eq} with $u(0,x;\hat{\psi}) = \hat{\psi}(x)$.
Theorem~\ref{globally-positive-existence}(iii) gives that
$\norm{u(t,\cdot;\hat{\psi})} \to 0$ as $t \to \infty$.  It then
follows from \eqref{additional-eq1} that
\begin{equation*}
\alpha \hat{\phi}(x) \le \phi(x)\quad \text{ for }\quad x \in D,
\end{equation*}
and then
\begin{equation*}
\hat{\psi}(x) \ge 0 \quad \text{ for } \quad x \in D.
\end{equation*}
This implies that
\begin{equation*}
\hat{\psi}(x) = 0 \quad \text{ for } \quad x \in D,
\end{equation*}
hence
\begin{equation*}
\alpha \hat{\phi}(x) = \phi(x) \text{ for } \quad x \in D.
\end{equation*}
Therefore we must have
\begin{align*}
&\sum_{i,j=1}^N a_{ij}(x) \int_{\Omega} \Bigl(
\frac{1}{w(\cdot)}\frac{\p w(\cdot)}{\p x_i} \Bigr)\,d\PP(\cdot)
\int_{\Omega} \Bigl( \frac{1}{w(\cdot)}\frac{\p w(\cdot)}{\p x_j}
\Bigr) \,d\PP(\cdot) \\
= &\sum_{i,j=1}^N a_{ij}(x) \int_{\Omega} \Bigl(
\frac{1}{w^2(\cdot)}\frac{\p w(\cdot)}{\p x_i} \frac{\p w(\cdot)}{\p
x_j}\Bigr)\,d\PP(\cdot)
\end{align*}
for all $x \in D$.

Let $\{\,x^{(n)}: n \in \NN\,\}$ be a countable dense subset of $D$.
By Lemma \ref{ch5-pre-holder-lm}(2), for each $n \in \NN$ there is
$\Omega^{(n)}$ with $\PP(\Omega^{(n)}) = 1$ such that $\disp
\frac{1}{w(\omega)(x^{(n)})}\frac{\p w(\omega)(x^{(n)}) }{\p x_i}$ is
independent of~$\omega \in \Omega^{(n)}$.  Consequently, from the
continuity of $\disp \frac{1}{w(\omega)(x)}\frac{\p w(\omega)(x) }{\p
x_i}$, for a fixed $\omega \in \Omega$, in $x \in D$, there are
$\Omega_5 \subset \Omega_0$, $\Omega_5 := \Omega_0 \cap
\bigcap_{n=1}^{\infty}\Omega^{(n)}$, with $\PP(\Omega_5) = 1$, and
functions $f_i(x)$ such that
\begin{equation*}
\frac{1}{w(\omega)(x)} \frac{\p w(\omega)(x) }{\p x_i} = f_i(x)
\end{equation*}
for $i = 1,2,\dots,N$, $\omega \in \Omega_5$ and $x \in D$.  Hence
\begin{equation*}
\nabla{\ln{w}}(\omega)(x) = (f_1(x),f_2(x),\dots,f_N(x))^\top
\end{equation*}
for $\omega \in \Omega_5$ and $x \in D$.  This implies that
$w(\omega)(x) = F(x) G(\omega)$ for some continuous $F(x) > 0$,
measurable $G(\omega) > 0$ and any $\omega \in \Omega_5$, $x \in D$.
Let $\Omega^{*} := \bigcap_{r\in\QQ}\theta_r\Omega_5$, where $\QQ$ is
the set of all rational numbers.  Clearly, $\PP(\Omega^{*}) = 1$ and
$w(\theta_t\omega)(x) = F(x)G(\theta_t\omega)$ for $t \in \QQ$,
$\omega \in \Omega^{*}$ and $x \in D$.  The continuity of
$w(\theta_t\omega)(x)$ in $t \in \RR$ then implies that the function
$[\, \RR \ni t \mapsto w(\theta_t\omega)(x)/F(x) \in \RR]$ is
continuous.  Hence, for each $\omega \in \Omega^{*}$ and each $t \in
\RR$ we can safely write $G(\theta_t\omega)$ for
$w(\theta_t\omega)(x)/F(x)$.  Therefore, by
\eqref{nonauton-hat-v-eq1},
\begin{align}
&\label{auxiliary}
F(x)\frac{dG(\theta_t\omega)}{dt} \nonumber\\
=\null & \Bigl( \sum_{i,j=1}^N a_{ij}(x) \frac{\p^2 F}{\p x_i \p x_j} +
\sum_{i=1}^N a_i(x) \frac{\p F}{\p x_i} + c(\theta_t\omega,x)F -
\kappa(\theta_t\omega)F\Bigr) G(\theta_t\omega)
\end{align}
for $t \in \RR$, $\omega \in \Omega^{*}$ and $x \in D$, and
\begin{equation*}
\mathcal{B}(\theta_t\omega) F = 0
\end{equation*}
for $t \in \RR$, $\omega \in \Omega^{*}$ and $x \in \p D$.  By
dividing both sides of~\eqref{auxiliary} by $F(x) G(\theta_t\omega)$
we obtain
\begin{equation*}
c(\theta_t\omega,x)  = \frac{\frac{dG(\theta_t\omega)}{dt}}
{G(\theta_{t}\omega)} + \kappa(\theta_{t}\omega) -
\frac{\sum_{i,j=1}^N a_{ij}(x) \frac{\p^2 F}{\p x_i \p x_j}(x) +
\sum_{i=1}^N a_i(x) \frac{\p F}{\p x_i}(x)}{F(x)}
\end{equation*}
for $t \in \RR$, $\omega \in \Omega^{*}$ and $x \in \p D$.  We can
write $c(\theta_{t}\omega,x) = c_{1}(x) + c_{2}(\theta_{t}\omega)$
for some integrable $c_{2}(\omega)$ with $\int_{\Omega}
c_{2}(\cdot)\,d\PP(\cdot) = 0$, any $x \in D$, $t \in \RR$ and
$\omega \in \Omega^{*}$.  Similarly, by taking the boundary condition
$\mathcal{B}(\theta_t\omega) F = 0$ we obtain that
\begin{equation*}
d(\theta_{t}\omega,x) = - \frac{\sum_{i=1}^{N}b_i(x)\frac{\p F}{\p
x_i}(x)}{F(x)}
\end{equation*}
for $t \in \RR$, $\omega \in \Omega^{*}$ and $x \in \p D$, that is,
$d(\theta_{t}\omega,x) = d(x)$ for each $x \in \p D$, each $t \in
\RR$ and each $\omega \in \Omega^{*}$.

(1) Let $\PP$ be the unique ergodic measure on $Y(c,d)$.  By
Lemma~\ref{ch5-pre-existence-avg-lm},
\begin{equation*}
\hat{c}(x) := \lim_{t\to\infty}\frac{1}{t}\int_{0}^{t}c(s,x)\,ds
\end{equation*}
exists for $x \in D$, and
\begin{equation*}
\hat{d}(x) := \lim_{t\to\infty}\frac{1}{t}\int_{0}^{t}d(s,x)\,ds
\end{equation*}
exists for $x \in \p D$.

Assume that the equality $\lambdainf(c,d) = \lambda(\hat{c},\hat{d})$
holds.  By Theorem \ref{ch5-smoothlb-thm2}(2), there is $Y_0(c,d)
\subset Y(c,d)$ with $\PP(Y_0(c,d)) = 1$ such that for any
$(\tilde{c},\tilde{d}) \in Y_0(c,d)$, $\tilde{c}(t,x) = c_{1}(x) +
\check{c}_2((\tilde{c},\tilde{d}) \cdot t)$ for some $\PP$-integrable
$\check{c}_2$ with $\int_{Y(c,d)}\check{c}_2(\tilde{c},\tilde{d})
\,d\PP(\tilde{c},\tilde{d}) = 0$, any $t \in \RR$, $x \in D$, and
$\tilde{d}(t,x) = d(x)$ for $x \in \p D$. Take a
$(\tilde{c},\tilde{d}) \in Y_0(c,d)$.  Since $Y(c,d)$ is minimal,
there is a sequence $(s_n)$ such that $(\tilde{c},\tilde{d}) \cdot
s_n$ converges in $Y(c,d)$ to $(c,d)$ as $n \to \infty$.  This
implies that $c(t,x) = c_{1}(x) + c_{2}(t)$ and $d(t,x) = d(x)$.  By
unique ergodicity, $\lim_{t\to\infty} \frac{1}{t}\int_{0}^{t}
c_{2}(s)\,ds = \int_{Y(c,d)}\check{c}_2(\tilde{c},\tilde{d})
\,d\PP(\tilde{c},\tilde{d}) = 0$.

If $c(t,x) = c_{1}(x)+ c_{2}(t)$ and $d(t,x) = d(x)$,  then the
equality $\lambdainf(c,d) = \bar{\lambda}$ follows clearly.
\end{proof}

\end{document}